\documentclass[12pt,reqno]{amsart}
\usepackage{amsmath,amssymb,mathrsfs,amsthm}
\usepackage{amsfonts}
\usepackage{braket}
\usepackage[inline]{enumitem} 
\usepackage{mathtools} 
\usepackage{nccbbb} 
\usepackage[utf8]{inputenc} 
\usepackage[sharp]{easylist}
\usepackage{hyperref}
\hypersetup{%
	colorlinks=true, linkcolor=blue, 
	citecolor=green
}

\newtheorem{theorem}{Theorem}[section]
\newtheorem{lemma}[theorem]{Lemma}
\newtheorem{proposition}[theorem]{Proposition}
\newtheorem{corollary}[theorem]{Corollary}
\theoremstyle{definition}
\newtheorem{remark}[theorem]{Remark}

\numberwithin{equation}{section}
\allowdisplaybreaks

\usepackage{acronym}

\acrodef{fbm}[$ \frac14 $-fbm]{$ \frac14 $-Fractional Brownian Motion}
\acrodef{aSHE}[ashe]{Additive Stochastic Heat Equation}
\acrodef{SPR}[spr]{Super-Polynomially Rate}
\acrodef{OP}[op]{Overwhelming Probability}
\acrodef{BDG}[bdg]{Burkholder--Davis--Gundy}

\usepackage{graphicx}
\newcommand*{\Cdot}{{\raisebox{-0.5ex}{\scalebox{1.8}{$\cdot$}}}} 

\newcommand{\BK}[1]{ {\left( #1 \right)} }
\newcommand{\sqBK}[1]{ {\left[ #1 \right]} }
\newcommand{\curBK}[1]{ {\left\{ #1 \right\}} }
\newcommand{\absBK}[1]{ {\left| #1 \right|} }
\newcommand{\VertBK}[1]{ {\left\Vert #1 \right\Vert} }
\newcommand{\BVert}{ \Big\Vert }
\newcommand{\Vertbk}[1]{ \Vert #1 \Vert }
\newcommand{\angleBK}[1]{ {\left < #1 \right >} }
\newcommand{\anglebk}[1]{ {\langle #1 \rangle} }

\newcommand{\ceilbk}[1]{ {\lceil #1 \rceil} }
\newcommand{\floorBK}[1]{ {\left\lfloor #1 \right\rfloor} }
\newcommand{\floorbk}[1]{ {\lfloor #1 \rfloor} }

\DeclareMathOperator{\sech}{sech}	
\DeclareMathOperator{\sign}{sign}	

\DeclareMathOperator{\Ex}{\mathbf{E}}	
\DeclareMathOperator{\Pro}{\mathbf{P}}	

\DeclareMathOperator{\ind}{\mathbf{1}}					
\DeclareMathOperator{\Exp}{Exp}							
\newcommand{\PPP}{\text{PPP}}							
\newcommand{\PPPp}{\text{PPP}_+(2\varepsilon^{-1/2})}	
\DeclareMathOperator{\Pois}{Pois}						

\newcommand{\pN}{p^\text{N}}	
\newcommand{\qN}{q^\text{N}}	
\newcommand{\pbfe}{p^{\boldsymbol{\epsilon}}}
\newcommand{\pNbfe}{p^{\text{N},\boldsymbol{\epsilon}}}
\newcommand{\qNbfe}{q^{\text{N},\boldsymbol{\epsilon}}}

\newcommand{\I}{I^\varepsilon} 

\newcommand{\N}{\bbZ_{+}}
\newcommand{\bbR}{ \mathbb{R} }
\newcommand{\bbZ}{ \mathbb{Z} }

\newcommand{\Qsp}{\mathscr{Q}}
\newcommand{\Qnorm}[1]{\left|#1\right|_{\mathscr{Q}}}
\newcommand{\QTnorm}[1]{\left|#1\right|_{\mathscr{Q}_T}}

\newcommand{\EM}{ Q^{\varepsilon} } 	
\newcommand{\EMz}{ Q^{\varepsilon,(0)} }	

\newcommand{\Pe}{ \widetilde{Q}^{\varepsilon} } 	
\newcommand{\Pek}{ \widetilde{Q}^{\varepsilon,k} } 	
\newcommand{\Qe}{ \widehat{Q}^{\varepsilon} } 	

\newcommand{\Ae}{A^\varepsilon}			
\newcommand{\Bul}{W^\varepsilon} 		
\newcommand{\Bulp}{W^{\varepsilon,*}} 	
\renewcommand{\Re}{R^\varepsilon}		

\newcommand{\flu}{ \mathcal{F}^{ \varepsilon,\delta } }
\newcommand{\flux}{ \mathcal{F}^{ \varepsilon,\varepsilon^{a} } }
\newcommand{\flug}{ \mathcal{G}^{ \varepsilon } }
\newcommand{\grd}{ \mathcal{\rho}^{\varepsilon} }

\newcommand{\icQ}{ \mathcal{W}^{ \boldsymbol{\epsilon} } }
\newcommand{\ic}{ \mathcal{W}^{\boldsymbol{\epsilon},*} }
\newcommand{\icc}{ \mathcal{W} }
\newcommand{\rd}{ \mathcal{R}^{\boldsymbol{\epsilon}} }
\newcommand{\ato}{ \mathcal{A}^{\boldsymbol{\epsilon}} }
\newcommand{\atoo}{ \mathcal{A}^{ \varepsilon,\delta,\e^{b} } }

\newcommand{\mg}{ \mathcal{M}^{\boldsymbol{\epsilon}} }
\newcommand{\mgg}{ \mathcal{M} }
\newcommand{\mgloc}{ \mathcal{N}^{\boldsymbol{\epsilon}} }

\newcommand{\tagpp}{ \widetilde{\mathcal{X}}^{\varepsilon} }
\newcommand{\tagp}{ \mathcal{X}^{\varepsilon} }
\newcommand{\tagP}{ \mathcal{X} }

\newcommand{\Fl}{ F^{\varepsilon, \text{l}} }	
\newcommand{\Flp}{ F^{\varepsilon,\text{l},*} }	
\newcommand{\Ul}{ U^{\varepsilon, \text{l}} }	
\newcommand{\Ur}{ U^{\varepsilon, \text{r}} }	
\newcommand{\Hl}{ H^{\varepsilon, \text{l}} }
\newcommand{\Hr}{ H^{\varepsilon, \text{r}} }

\newcommand{\Xl}{X^{\varepsilon,\text{l}}}
\newcommand{\Xlp}{X^{\varepsilon,\text{l},*}}
\newcommand{\Xr}{X^{\varepsilon,\text{r}}}

\newcommand{\te}{t^\varepsilon} 

\newcommand{\tloc}{\widehat{t_\varepsilon}}
\newcommand{\tploc}{\widehat{t'_\varepsilon}}

\newcommand{\e}{\varepsilon}
\newcommand{\bfe}{\boldsymbol{\epsilon}}
\renewcommand{\ne}{ n_{\boldsymbol{\epsilon}} }

\newcommand{\Dbfe}{ D^{\boldsymbol{\epsilon}} }
\newcommand{\Fbfe}{ F^{\boldsymbol{\epsilon}} }
\newcommand{\Gbfe}{ G^{\boldsymbol{\epsilon}} }
\newcommand{\GbfeL}{ G^{\boldsymbol{\epsilon},L} }

\newcommand{\mbfe}{ m^{\boldsymbol{\epsilon}} }
\newcommand{\Nbfe}{ N^{\boldsymbol{\epsilon}} }
\newcommand{\fbfe}{ f^{\boldsymbol{\epsilon}} }
\newcommand{\gbfe}{ g^{\boldsymbol{\epsilon}} }

\newcommand{\rbfe}{ r^{\boldsymbol{\epsilon}} }

\newcommand{\Psibfe}{ \Psi^{\boldsymbol{\epsilon}} }
\newcommand{\Psitilbfe}{ \widetilde{\Psi}^{ \boldsymbol{\epsilon} } }
\newcommand{\scrFbfe}{ \mathscr{F}^{\boldsymbol{\epsilon}} }

\newcommand{\bfz}{\mathbf{0}}
\newcommand{\bfv}{\mathbf{v}}

\newcommand{\bfic}{\mathbf{W}^{\boldsymbol{\epsilon},*}}

\newcommand{\cha}{\varphi}	

\newcommand{\phie}{\phi^\varepsilon}
\newcommand{\taue}{\tau^\varepsilon}	

\newcommand{\de}{ d^{\varepsilon} }
\newcommand{\fe}{f^\varepsilon}

\newcommand{\ie}{i_\varepsilon} %
\newcommand{\re}{r^\varepsilon}
\newcommand{\xe}{x^\varepsilon}

\newcommand{\Be}{B^\varepsilon} 
\newcommand{\De}{D^\varepsilon}

\newcommand{\Fe}{F^\varepsilon} %
\newcommand{\Ge}{G^\varepsilon}
\newcommand{\He}{H^\varepsilon}
\newcommand{\Je}{J^\varepsilon}
\newcommand{\Se}{S^\varepsilon}

\newcommand{\Me}{M^\varepsilon} 	
\newcommand{\Xe}{X^\varepsilon} 	
\newcommand{\Ye}{Y^\varepsilon} 	

\newcommand{\calDe}{\mathcal{D}^\varepsilon}	%

\newcommand{\barB}{\overline{B}}

\newcommand{\calA}{ \mathcal{A} }

\newcommand{\calK}{ \mathcal{K} }
\newcommand{\calN}{ \mathcal{N} }

\newcommand{\calV}{ \mathcal{V} }

\newcommand{\scrB}{ \mathscr{B} }

\newcommand{\scrI}{ \mathscr{I} }
\newcommand{\scrK}{ \mathscr{K} }

\newcommand{\scrW}{ \mathscr{W} }

\begin{document}

\title[Fluctuation of the Atlas Model]{
Equilibrium fluctuation of\\ 
the Atlas Model}

\author[A. Dembo]{Amir Dembo$^\star$}
\address{Departments of Statistics and Mathematics \\ Stanford University, Stanford, CA 94305}
\email{adembo@stat.stanford.edu}
\author[L.-C. Tsai]{Li-Cheng Tsai$^\dagger$}
\address{Department of Mathematics \\ Stanford University, Stanford, CA 94305}

\thanks{$^*$Research partially supported by NSF grant \#DMS-1106627} 
\thanks{$^\dagger$Research partially supported by NSF grant \#DMS-0709248}
\date{\today}

\begin{abstract}
We study the fluctuation of the Atlas model,
where a unit drift is assigned to the lowest ranked particle
among a semi-infinite ($ \N $-indexed) 
system of otherwise independent Brownian particles,
initiated according to a Poisson point process on $ \bbR_+ $. 
In this context, we show that the joint law of ranked particles,
after being centered and scaled by $t^{-1/4}$, 
converges as $t \to \infty$ to the Gaussian field corresponding to
the solution of the \ac{aSHE} on $\bbR_+$ 
with Neumann boundary condition at zero.
This allows us to express the asymptotic fluctuation 
of the lowest ranked particle in terms of a \ac{fbm}.
In particular, we prove a conjecture of Pal and Pitman \cite{pal08} about 
the asymptotic Gaussian fluctuation of the ranked particles.
\end{abstract}

\maketitle

\section{Introduction}
\label{sect:intro}

In this paper we study the infinite particles Atlas model.
That is, we consider the $ \bbR^{\N} $-valued process $\{X_i(t)\}_{i\in\N}$,
each coordinate performing an independent Brownian motion
except for the lowest ranked particle receiving a drift of strength $\gamma>0$.
For suitable initial conditions, this process 
is given by the unique weak solution of
\begin{align}\label{eq:CBP}
	dX_i(t) = \gamma \ind_\curBK{X_i(t)=X_{(0)}(t)} dt + dB_i(t),
	\quad
	i\in\N.
\end{align}
Hereafter 
$ B_i(t) $, $ i\in\bbZ_+ $, denote independent standard Brownian motions 
and $X_{(i)}(t)$, $ i\in\bbZ_+ $, denote the \emph{ranked} particles,
i.e.\ $ X_{(0)}(t) \leq X_{(1)}(t) \leq \ldots $.
More precisely,
recall that $ (x_i)\in\bbR^{\N} $ is rankable if
there exists a bijection $ \pi: \N\to\N $ (i.e.\ permutation)
such that $ x_{\pi(i)} \leq x_{\pi(j)} $ for all $ i\leq j \in\N $.
Such ranking permutation is unique up to ties, which we break in lexicographic order.
The equation \eqref{eq:CBP} is then well-defined
if $ (X_i(t))_{i\in\N} $ is rankable at all $ t\geq 0 $
with a measurable ranking permutation.

The Atlas model \eqref{eq:CBP}
is a special case of diffusions with \emph{rank dependent drifts}.
In finite dimensions,
such systems are studied in \cite{bass87},
motivated by questions in filtering theory,
and in \cite{fernholz02,karatzas09},
in the context of stochastic portfolio theory.
See also \cite{chatterjee10,chatterjee11,ichiba10,ichiba13,ichiba11},
for their ergodicity and sample path properties,
and \cite{dembo12,pal14}
for their large deviations properties as the dimension tends to infinity.
The Atlas model is a simple special case
(where the drift vector is specialized to $ (\gamma,0,\ldots,0) $)
that allows more detailed analysis.
In particular, Pal and Pitman \cite{pal08}
consider the infinite dimensional Atlas model \eqref{eq:CBP},
establishing well-posedness and the existence of an explicit invariant measure,
see also \cite{ichiba13,shkolnikov11}.

In this paper we study the long-time behavior of the ranked particles, 
in particular the lowest ranked particle.
This amounts to understanding competition between
the drift $ \gamma $ and the push-back from the bulk of particles (due to ranking).
These two effects act against each other,
and balance exactly at the critical density $ 2\gamma $.
More precisely, recall from \cite{pal08} that,
starting from $ \{X_{(i)}(0)\} \sim \PPP_+(2\gamma) $,
the Poisson Point Process with density $ 2\gamma $ on $ \bbR_+:=[0,\infty) $,  
\eqref{eq:CBP} admits a unique weak solution (which is rankable)
such that $ \{X_{(i)}(t)-X_{(0)}(0)\}_{i\in\N}$ retains the $\PPP_+(2\gamma) $ law
for all $ t\geq 0 $.
At this critical density,
we show that, for large $ t $ and for all $ i $, $ X_{(i)}(t) $
fluctuates at $ O(t^{1/4}) $,
and the joint law of the fluctuations of the particles
scales to a Gaussian field characterized by \ac{aSHE}.

Hereafter we fix $ \{X_i(t)\}_{i\in\N} $ 
to be the unique weak solution of \eqref{eq:CBP}
starting from $ \PPP_{+}(2\gamma) $.
With $ Y_i(t):= X_{(i+1)}(t) - X_{(i)}(t) $ denoting the $ i $-th gap,
such initial condition are equivalent to 
$ X_{(0)}(0)=0 $ and $ \{Y_i(0)\}_{i\in\N} \sim \bigotimes_{i\in\N}\Exp(2\gamma) $.
We consider the process
\begin{align}\label{eq:xi}
	\tagp_t(x) := 
	\e^{1/4} \sqBK{ \ie(x) - 2\gamma X_{(\ie(x))}(\e^{-1}t) },
	\quad
	\ie(x) := \floorBK{ (2\gamma \e^{1/2})^{-1} x }.
\end{align}
Recall that the the relevant solution of the \ac{aSHE}, \eqref{eq:SHE},
is invariant under the scaling
$ \tagP_t(x) \mapsto a^{1/4} \tagP_{t/a}(x/a^{1/2}) $,
which suggests the scaling of \eqref{eq:xi}.
Alternatively, this scaling can be understood as
choosing the diffusive scaling of $ (t,x) $ to respect $ B_i(\Cdot) $,
and choosing the $ \e^{1/4} $ factor to capture the Gaussian
fluctuation of $ \PPP_+(2\gamma\e^{-1/2}) $.

Let
$ p(x) = \Phi'(x) = (2\pi)^{-1/2} e^{-x^2/2} $ be the standard Gaussian density,
with $ p_t(x):=p(xt^{-1/2}) $ the heat kernel
and $\Phi_t(x)= \Phi(xt^{-1/2})$ the scaled error function.
We use $\pN_t(y,x):=p_t(y-x)+p_t(y+x)$ for the Neumann heat kernel
and
\begin{align}\label{eq:Psi}
	\Psi_t(y,x):= 2-\Phi_t(y-x)-\Phi_t(y+x).
\end{align}
Hereafter we endow the space of right-continuous-left-limit functions
on $ \bbR_+^2 $ the topology of uniform convergence on compact sets,
and use $ \Rightarrow $ to denote weak convergence of probability measures.
Our main result is as follows.

\begin{theorem}\label{thm:main}
Let $\tagP_\Cdot(\Cdot)$ denote the $C(\bbR_+^2,\bbR)$-valued 
centered Gaussian process with covariance
\begin{align}\label{eq:cov}
\begin{split}
	&
	\Ex \BK{ \tagP_{t}(x) \tagP_{t'}(x') }
\\
	&
	\quad 
	=
	2\gamma
	\BK{
		\int_0^\infty \Psi_{t}(y,x) \Psi_{t'}(y,x') dy
		+
		\int_0^{t \wedge t'} 
		\int_0^\infty \pN_{t-s}(y,x) \pN_{t'-s}(y,x') dy ds
	}.
\end{split}
\end{align}
Then, $\tagp_\Cdot(\Cdot)\Rightarrow \tagP_\Cdot(\Cdot)$, as $ \e\to 0 $.
\end{theorem}

\begin{remark}\label{rmk:SHE}
The limiting process $ \tagP_\Cdot(\Cdot) $ can be equivalently 
characterized by the solution
of the \ac{aSHE} on $ \bbR_+ $,
\begin{align}\label{eq:SHE}
	\BK{ \partial_t - \frac{1}{2} \partial_{xx} } \tagP_t(x) 
	= (2\gamma)^{1/2} \dot{\scrW},
	\quad
	t, x >0,
\end{align}
with the initial condition $ \tagP_0(x) = (2\gamma)^{1/2}B(x) $
and a suitable boundary condition.
Here $ B(x) $ denotes a standard Brownian motion
and $ \scrW(t,x) $ denotes a $ 2 $-dimensional white noise,
independent of $ B(\Cdot) $.
In the course of proving Theorem~\ref{thm:main},
extracting the boundary condition requires 
a \emph{special choice} of the test function (see \eqref{eq:flu}).
From this, we end up with the \emph{Neumann boundary condition}.
That is, we declare the semi-group of \eqref{eq:SHE} to be $ \pN_t(y,x) $,
whereby obtaining $ \tagP_t(x) = \icc_t(x) + \mgg_t(x) $,
for
\begin{align}
	&
	\label{eq:icc}
	\icc_t(x) :=
	\int_0^\infty \pN_t(y,x) \tagP_0(y) dy	
	= 
	(2\gamma)^{1/2} 
	\int_0^\infty \Psi_t(y,x) dB(y),	
\\
	&
	\label{eq:mgg}
	\mgg_t(x)
	:=  
	(2\gamma)^{1/2} \int_0^t \int_0^\infty \pN_{t-s}(y,x) d\scrW(y,s).
\end{align}
The former and latter, measurable with respect to
$ B(\Cdot) $ and $ \scrW(\Cdot,\Cdot) $, respectively, are independent.
From \eqref{eq:icc} and \eqref{eq:mgg},
one then concludes the covariance as given in \eqref{eq:cov}.

In retrospect, the Neumann boundary condition represents the conservation of particles at $ x=0 $.
It is shown in \cite{cabezas15} that at the equilibrium density we consider here,
$ \sup_{s\in[0,t]}\{\e^{1/2} |X_{(0)}(\e^{-1}t)| \} \to 0 $ almost surely.
That is, at the scale $ \e^{-1/2} $ of space, the lowest rank particle stays very close to $ x=0 $.
Consequently, the flux at $ x=0 $ should be zero, which amounts to the Neumann boundary condition.
\end{remark}

\begin{remark}
If starting \eqref{eq:CBP} at deterministic equi-distant particle positions, 
i.e.\ $ X_{(i)}(0) = 2\gamma i $, 
one should naturally expect to end up 
with the limiting process $ \tagP_t(x) := \mgg_t(x) $
(corresponding to $ \tagP_0(x) = 0 $).
However, our proof of Theorem~\ref{thm:main} relies 
on the stationarity of $ \{X_{(i)}(\Cdot)-X_{(0)}(\Cdot)\} $
to simplify a-priori estimates,
and hence does not apply to this deterministic initial condition.
\end{remark}

An important consequence of Theorem~\ref{thm:main} is:
\begin{corollary}\label{cor:PP}
\begin{enumerate}[label=(\alph*)]
	\item[]
	\item \label{enu:fmb}
	Let $ B^{(H)}(\Cdot) $ denote the fractional Brownian motion 
	with Hurst parameter $ H $.
	As $ \e\to 0 $,
	$ \e^{-1/4} X_{(0)}(\e^{-1} \Cdot) $,
	the scaled fluctuation of the lowest ranked particle,
	weakly converges to $ (2/\pi)^{1/4}\gamma^{-1/2} B^{(1/4)}(\Cdot) $.
	\item \label{enu:PP}
	As $ t\to\infty $, $ (X_{(\ie(x))}(t) - X_{(\ie(x))}(0)) t^{-1/4} $
	weakly converges to a centered Gaussian with variance $ \sigma^2(x) $,
	satisfying $ \sigma(0)= (2/\pi)^{1/4} \gamma^{-1/2} $
	and $ \lim_{x\to\infty} \sigma(x) = (2\pi)^{-1/4} \gamma^{-1/2} $.
\end{enumerate}
\end{corollary}
\noindent
Indeed, it is not difficult to deduce from \eqref{eq:cov} the covariance of
the center Gaussian process
$ 
	\calK_\Cdot(x) :=
	(2\gamma)^{-1}(\tagP_\Cdot(x) - \tagP_0(x))
$
for the special case of $ x=0 $ and $ x\to\infty $,
and to arrive at
\begin{align}
	&
	\label{eq:cov:tag0}
	\Ex\BK{ \calK_t(0) \calK_{t'}(0) } =
	\gamma^{-1}\BK{2\pi}^{-1/2} \BK{ t^{1/2} + (t')^{1/2} - |t-t'|^{1/2} },
\\
	&
	\label{eq:cov:tag00}
	\lim_{x\to\infty} \Ex\BK{ \calK_t(x) \calK_{t'}(x) }
	=
	\gamma^{-1}\BK{8\pi}^{-1/2} \BK{ t^{1/2} + (t')^{1/2} - |t-t'|^{1/2} }.
\end{align}
From \eqref{eq:cov:tag0}--\eqref{eq:cov:tag00} Corollary~\ref{cor:PP} immediately follows.

Theorem~\ref{thm:main} is the first result of asymptotic fluctuations of \eqref{eq:CBP},
with Corollary~\ref{cor:PP}\ref{enu:PP}
resolving the conjecture of Pal and Pitman \cite[Conjecture 3]{pal08}.
Further, Theorem~\ref{thm:main} establishes 
the previously undiscovered connection of \eqref{eq:CBP} to \ac{aSHE}.

\begin{remark}
In \cite{cabezas15}, the hydrodynamic limits of the Atlas model \eqref{eq:CBP} is studied. 
For out-of-equilibrium initial conditions,
it is shown that $ \e^{1/2} X_{(0)}(\e^{-1}\Cdot) $
converges to a deterministic limit described by the one-sided Stefan's problem.
For the symmetric simple exclusion process on $ \bbZ $,
\cite{landim98} shows that the hydrodynamic limit of 
a tagged particle is described by the two-sided Stefan's problem.
For the same model, \cite{landim00} shows that the fluctuation scales to 
a generalized Ornstein--Uhlenbeck process related to \ac{aSHE}.
\end{remark}

\begin{remark}\label{rmk:harris}
Harris \cite{harris65} introduces a closely related model
of i.i.d.\ $ \bbZ $-indexed Brownian particles $ B_i(t) $,
which can be regarded as the bulk version of \eqref{eq:CBP}.
Using an explicit formula for the law of $ B_{(0)}(t) $,
he shows that at equilibrium with density $ 2\gamma $,
$ 
	\lim_{t\to \infty} t^{-1/4} ( B_{(0)}(t) - B_{(0)}(0) ) 
	\Rightarrow 
	(2\pi)^{-1/4} \gamma^{-1/2} B(1).
$
This result is further extended by  \cite{duerr85}
to the functional convergence 
$ 
	\e^{1/4} ( B_{(0)}(\e^{-1}\Cdot) - B_{(0)}(0) ) 
	\Rightarrow
	(2\pi)^{-1/4} \gamma^{-1/2} B^{(1/4)}(\Cdot) 
$.

Intuitively,
we expect the Atlas model to behavior similarly to the Harris model 
once we match the equilibrium density.
This is indeed confirmed in \eqref{eq:cov:tag00}.
That is, at the bulk ($ x\to\infty $)
the asymptotic fluctuation of the two systems are approximately equal,
to $ (2\pi)^{-1/4} \gamma^{-1/2} B^{1/4}(\Cdot) $.
Somewhat unexpectedly, as shown in Corollary~\ref{cor:PP}\ref{enu:fmb},
the \ac{fbm} fluctuation also appears at $ x=0 $, 
but with a \emph{different} prefactor.
\end{remark}

\begin{remark}
Applying our technique to the Harris model,
one may rederive the results of \cite{duerr85,harris65}.
This provides an explanation of the scaling and the \ac{fbm} limit
as the fluctuation of \ac{aSHE} at $ x=0 $.
Specifically, the scaling limit of the Harris model 
should be \ac{aSHE} on $ \bbR $ with no boundary condition.
Since no drift presents in the Harris model,
the latter scaling limit 
could be deduced directly from the time evolution equation.
\end{remark}

Our strategy of proving Theorem~\ref{thm:main} is to focus on the empirical measure.
While this strategy has been widely used for interacting particle systems,
in the context of Atlas model, 
or more generally diffusions with rank-dependent drifts,
analyzing empirical measure is a new approach 
that has only been used here and in \cite{cabezas15}.
It completely bypasses the need of local times,
which is a major a challenge when analyzing diffusions with rank-dependent drifts.

To define the empirical measure,
we consider $ w(y):=e^{-y}\wedge 1 $, 
$\Qnorm{\phi} := \sup_{y\in\bbR} |\phi(y)|/w(y)$,
and
$
	\Qsp := \curBK{ \phi\in L^\infty(\bbR) : 
	\Qnorm{\phi(y)}<\infty }.
$
Let $ \Xe_{i}(t) := \e^{1/2} X_{i}(\e^{-1}t) $,
$ \Xe_{(i)}(t) := \e^{1/2} X_{(i)}(\e^{-1}t) $
and, for any $ \phi\in\Qsp $, let
\begin{align}
	&
	\label{eq:EM}
	\angleBK{\EM_t,\phi} 
	:= 
	\sum_{i=0}^\infty \phi(\Xe_i(t)),
\\
	&
	\label{eq:Q}
	\angleBK{\Qe_t,\phi}
	:=
	\e^{1/4} 
	\BK{ 
		\angleBK{ \EM_t, \phi } - 2\gamma \e^{-1/2} \int_0^\infty \phi(y) dy
		},
\end{align}
which are well-defined (see Lemma~\ref{lem:LebBd}).
As we are at equilibrium, $ \EM_t $ is a $ \PPP_+(2\gamma\e^{-1/2}) $ translated by $ \Xe_{(0)}(t) $,
so $ \Qe_t $ captures the Gaussian fluctuation of $ \PPP_+(2\gamma\e^{-1/2}) $
around $ 2\gamma\e^{-1/2} \ind_{\bbR_+}(y) dy $.

Under this framework,
the main challenge of proving Theorem~\ref{thm:main} is to choose
the test function $ \flu_t(x) $ that 
\begin{enumerate*}[label=\itshape\roman*\upshape)]
\item identifies the relevant boundary condition;
and \item relates itself to the process $ \tagp_t(x) $.
\end{enumerate*}
With
\begin{align}
	\label{eq:flu}
	\flu_t(x) := \anglebk{\Qe_t,\Psi_\delta(\Cdot,x)},
\end{align}
establishing (\textit{ii})
amounts to approximating the displacement of a ranked particle
by the net flux of particles,
which we achieve by using stationarity.
In Sections~\ref{sect:fluConv} and \ref{sect:fluxi} we prove
Propositions~\ref{prop:fluConv} and \ref{prop:fluxi}, respectively, 
from which Theorem~\ref{thm:main} follows immediately.

\begin{proposition}\label{prop:fluConv}
Fix any $ b\in(0,1/4) $.
As $ (\e,\delta)\to (0,0) $, 
$\flu_\Cdot(\Cdot+\e^{b})\Rightarrow \tagP_\Cdot(\Cdot)$,
for $ \tagP_t(x) $ given as in Theorem~\ref{thm:main}.
\end{proposition}
\begin{proposition}\label{prop:fluxi}
Fix any $ a\in(1/2,\infty) $ and $ b\in(0,1/4) $.
As $ \e\to 0 $,
$ \flux_\Cdot(\Cdot+\e^{b}) - \tagp_\Cdot(\Cdot)\Rightarrow 0 $.
\end{proposition}
\section{Outline of the Proof of Propositions~\ref{prop:fluConv} and \ref{prop:fluxi}}
\label{sect:outline}

Without lost of generality,
we scale the drift $ \gamma>0 $ to unity 
by $ X_i(t) \mapsto \gamma X_i(\gamma^{-2}t) $.
Hereafter,
we \emph{fix $ \gamma:=1 $} and use $ C(a,b,\ldots) $ 
to denote generic positive finite (deterministic) constant
that depends only on the designated variables.

We proceed to describe the time evolution of $ \Qe_t $.
To this end, let
\begin{align*}
	&
	\Qsp_T :=
	\curBK{
		\psi_\Cdot(\Cdot) \in C^2\BK{ \bbR\times[0,T] }:
		\QTnorm{\psi} < \infty
		},
\\
	&
	\QTnorm{ \psi }
	:=
	\sup_{t\in[0,T]} 
	\BK{
		\Qnorm{ \partial_t \psi_t }
		+ \Qnorm{ \partial_{x} \psi_t }
		+ \Qnorm{ \partial_{xx} \psi_t }
		+ \Qnorm{ \psi_t }
	}.
\end{align*}
We decompose $ \Qe_t = \Ae_t+\Bul_t $, where
\begin{align}
	\label{eq:A}
	\angleBK{\Ae_t,\phi} := -2\e^{-1/4} \int_0^{\Xe_{(0)}(t)} \phi(y) dy
\end{align}
records the fluctuation of the lowest ranked particle,
and
\begin{align}
	\label{eq:V}
	\angleBK{\Bul_t,\phi} 
	:= 
	\e^{1/4} \BK{
		\angleBK{ \EM_t, \phi }
		-2 \e^{-1/2} \int_{\Xe_{(0)}(t)}^\infty \phi(y) dy
		}
\end{align}
accounts for the fluctuations of the bulk of particles.
For any $ \psi\in\Qsp_T $ and $ t_0\in[0,T] $, let
\begin{align}\label{eq:Mek}
	\Me_{t_0,t}(\psi,k) := \sum_{i=0}^k \int_{t_0}^t \partial_y\psi_s\BK{\Xe_{i}(s)} d\Be_i(s),
\end{align}
which is a $ C([t_0,T],\bbR) $-valued martingale in $ t $.

\begin{proposition}\label{prop:Ito}
For any $ T\in\bbR_+ $, $ t_0\in[0,T] $ and $\psi\in\Qsp_T$,
there exists a $ C([t_0,T],\bbR) $-valued martingale $ \Me_{t_0,\Cdot}(\psi,\infty) $ such that,
for all $ q\geq 1 $, 
\begin{align}\label{eq:Mcon}
	\BVert \sup_{t\in[t_0,T]} |\Me_{t_0,t}(\psi,k)-\Me_{t_0,t}(\psi,\infty)| \BVert_q \to 0.
\end{align}
Furthermore, almost surely
\begin{align}\label{eq:Ito}
\begin{split}
	&
	\angleBK{\Qe_t,\psi_t} - \angleBK{\Qe_0,\psi_0}
\\
	&\quad
	= \int_0^t \angleBK{ \Bul_s, \BK{\partial_s+\frac12 \partial_{yy}} \psi_s } ds
	+ \int_0^t \angleBK{ \Ae_s, \partial_s \psi_s} ds
	+ \Me_{0,t}(\psi,\infty),
\end{split}
\end{align}
for all $t\in[0,T]$.
\end{proposition}
\begin{remark}
Proposition~\ref{prop:Ito} is established in Section~\ref{sect:ApriEst},
where we derive \eqref{eq:Ito} via Ito calculus.
In this derivation, the driving Brownian motions $ B_i(t) $, $ i\in\N $,
collectively contribute
\begin{align*}
	\BK{
		\e^{1/4} \angleBK{\EM_t,(\partial_t+2^{-1}\partial_{yy})\psi_t}
		- 2\e^{-1/4} \int_0^\infty \partial_s\psi_s(y) dy
		}dt
	+ d\Me_{0,t}(\psi,\infty)
\end{align*}
whereas the drift $ \gamma=1 $ at the lowest ranked particle contributes 
\begin{align*}
	\e^{-1/4} \partial_y\psi_s(\Xe_{(0)}(t)) dt
	=
	\BK{ - \e^{-1/4} \int_{\Xe_{(0)}(t)}^\infty \partial_{yy}\psi_s(y) dy } dt.
\end{align*}
These, when combined together, give the expression \eqref{eq:Ito}.
\end{remark}

Based on Proposition~\ref{prop:Ito},
in Section~\ref{sect:ApriEst} we establish the following 
a-priori estimate of $ \Xe_{(0)}(\Cdot) $.

\begin{proposition}\label{prop:Xotig}
Fixing any $ q>1 $, $ b\in[0,1/4) $ and $ T\in\bbR_+ $, 
we let $\taue_b:=\inf\{t\geq 0 : |\Xe_{(0)}(t)| \geq \e^{b} \}$.
There exists $C=C(T,b,q)<\infty$ such that, 
for all $\e\in(0,(2q)^{-2}]$,
\begin{align}\label{eq:taub}
	\Pro \BK{ \taue_b \leq T } \leq C \e^{(1/4-b)q-1}.
\end{align}
\end{proposition}

\begin{remark}
Proposition~\ref{prop:Xotig} implies,
for any $ T\in\bbR_+ $ and $ b\in(0,1/4) $,
we have \\
$ \Pro( \sup_{t\in[0,T]} |\Xe_{(0)}(t)| \leq \e^{b} ) \to 1 $.
This is almost optimal,
since we know a-posteriori from Theorem~\ref{thm:main} that
$ \e^{-1/4}\Xe_{(0)}(t) = \tagp_t(x) $ converges weakly.
\end{remark}

Turning to the proof of Proposition~\ref{prop:fluConv},
for each $t, \delta,\eta>0 $, $ x\in\bbR_+ $, we apply
Proposition~\ref{prop:Ito} for 
$\psi_{s}(y):=\Psi_{t+\delta-s}(y,x+\eta) \in \Qsp_t$.
With $ \psi_s(y) $ solving the backward heat equation
$(\partial_s+2^{-1}\partial_{yy})\psi_s=0$, 
one easily obtains that
\begin{align*}
	\flu_t(x+\eta) = 
	\icQ_{t}(x) + \mg_{t}(x) + \ato_t(x),
\end{align*}
where $ \bfe := (\varepsilon,\delta,\eta) $,
\begin{align}
	&
	\label{eq:PsipNbfe}
	\Psibfe_t(y,x) := \Psi_{t+\delta}(y,x+\eta),
	\quad \pNbfe_t(y,x) := \pN_{t+\delta}(y,x+\eta),
\\
	&
	\label{eq:icQ}
	\icQ_t(x)
	:= \angleBK{ \Qe_0, \Psibfe_{t}(\Cdot,x) },
\\
	&
	\label{eq:mg}
	\mg_{t}(x) 
	:= \Me_{0,t}(\Psibfe_{\Cdot}(\Cdot,x),\infty)
	= 
	\e^{1/2} \sum_{i=0}^\infty 
	\int_{0}^{t} \pNbfe_{t-s}(\Xe_{i}(s),x) d\Be_i(s),
\\
	&
	\label{eq:ato}
	\ato_t(x)
	:=
	\int_0^t \angleBK{ \Ae_s, \partial_s\Psibfe_{t-s}(\Cdot,x) } ds.
\end{align}
Since $ \icQ_{t}(x) $ and $ \mg_t(x) $,
consisting respectively of the contribution of 
$ \{\Xe_i(0)\} $ and $ \{\Be_i(\Cdot)\} $, are independent,
Proposition~\ref{prop:fluConv} is an immediate consequence of:
\begin{proposition}\label{prop:fluConvv}
\begin{enumerate}[label=(\alph*)]
\item[]
\item	
	\label{enu:conato} 
	Fix any $ b\in(0,1/4) $.
	As $ (\e,\delta)\to (0,0) $, $ \atoo_\Cdot(\Cdot) \Rightarrow 0$.
\item
	\label{enu:conic} 
	As $ \bfe\to\bfz $, 
	$ \icQ_{\Cdot}(\Cdot) \Rightarrow \icc_{\Cdot}(\Cdot) $,
	where $ \icc_\Cdot(\Cdot) $ is a centered Gaussian process with
	\begin{align}\label{eq:icCov}
		\Ex \BK{ \icc_t(x) \icc_{t'}(x') }
		= 2 \int_0^\infty \Psi_t(y,x) \Psi_{t'}(y,x') dy.
	\end{align}
\item
	\label{enu:conmg} 
	As $ \bfe\to\bfz $, 
	$ \mg_{\Cdot}(\Cdot) \Rightarrow \mgg_\Cdot(\Cdot) $,
	where $ \mgg_\Cdot(\Cdot) $ is a centered Gaussian process with
	\begin{align}\label{eq:mgCov}
		\Ex \BK{ \mgg_t(x) \mgg_{t'}(x') }
		= 2 \int_0^{t\wedge t'} \int_0^\infty \pN_{t-s}(y,x) \pN_{t'-s}(y,x') dy ds.
	\end{align} 
\end{enumerate}
\end{proposition}

\begin{remark}\label{rmk:ato}
Our special choice of $ \psi_s(y) $ is what makes
Proposition~\ref{prop:fluConvv}\ref{enu:conato} valid.
To see this,
note that $ \Xe_{(0)}(t) = O(\e^b) $ for all $ b\in(0,1/4) $
(by Proposition~\ref{prop:Xotig})
and that $ \ato_t(x) = \int_0^t \anglebk{\Ae_s,\xi_s} ds $
for $ \xi_s(y) = \partial_s\Psi_{t+\delta-s}(y,x) $.
With $ \xi_s(0)=0 $,
by \eqref{eq:A} we can approximate $ \anglebk{\Ae_s,\xi_s } $ by 
$ \e^{-1/4} O((\Xe_{(0)}(s))^2)$, which indeed tends to zero.
Further,
we expect Proposition~\ref{prop:fluConvv}\ref{enu:conic}
and \ref{enu:conmg} to hold
by comparing \eqref{eq:icc} with \eqref{eq:icQ}, 
and \eqref{eq:mgg} with \eqref{eq:mg}, 
since $ \Qe_0 $ approximates $ 2dB_0(\Cdot) $, 
and $ \e^{1/2}\EM_t $ approximates $ 2\ind_{\bbR_+}(x)dx $, respectively.
\end{remark}

For the proof of Proposition~\ref{prop:fluxi},
we require the following notations:
\begin{align}
	&
	\label{eq:flug}
	\flug_t(x) := \anglebk{ \Qe_t, \ind_{(-\infty,x]} }
	= \e^{1/4} \angleBK{\EM_t,\ind_{(-\infty,x]}} - 2\e^{-1/4}x,
\\
	&
	\label{eq:I}
	\I_t(x) := \inf\curBK{i\in\N: \Xe_{(i)}(t) > x } = \angleBK{ \EM_t, \ind_{(-\infty,x]} },
\\
	&
	\label{eq:zetae}
	\tagpp_t(x) := 
	\e^{1/4} \BK{ \I_0(x) - 2 X_{(\I_0(x))}(\e^{-1}t) }.
\end{align}
Up to a centering and scaling,
$ \flug_t(x) $ counts the total number of particles to the left of $ x $,
and $ \tagpp_t(x) $ records the trajectory of $ X_{(\I_0(x))}(\Cdot) $,
where $ \Xe_{(\I_0(x))}(0) $ the first particle to the right of $ x $ at time $ 0 $.
Proposition~\ref{prop:fluxi} is then an immediate consequence of:

\begin{proposition}\label{prop:fluxii}
Let $ a\in(1/2,\infty) $ and $ b\in(0,1/4) $.
\begin{enumerate}[label=(\alph*)]
\item
	\label{enu:fluflug}
	As $ \e\to 0 $, 
	$ \flux_\Cdot(\Cdot+\e^{b}) - \flug_\Cdot(\Cdot+\e^{b}) \Rightarrow 0$.
\item 
	\label{enu:flugtagpp}
	As $ \e\to 0 $,
	$ \flug_{\Cdot}(\Cdot+\e^{b}) - \tagpp_\Cdot(\Cdot+\e^{b}) \Rightarrow 0$.
\item
	\label{enu:tagpptagp} 
	As $ \e\to 0 $,
	$ \tagpp_{\Cdot}(\Cdot+\e^b) - \tagp_\Cdot(\Cdot) \Rightarrow 0$.
\end{enumerate}
\end{proposition}

Letting
\begin{align}
	&
	\label{eq:rho}
	\rho^\e_t(x) := X_{(\I_t(x))}(\e^{-1} t) - \e^{-1/2} x 
	= \e^{-1/2} \sqBK{ \Xe_{(\I_t(x))}(t) - x },
\\	
	&
	\label{eq:calDX}
	\calDe(j,j',t) :=  j - j' - 2\BK{ X_{(j)}(\e^{-1}t) - X_{(j')}(\e^{-1}t) }
\\
	&
	\label{eq:calDY}
	\hphantom{\calDe(j,j',t)}
	=
	\sign(j-j')
	\sum_{i\in[j',j)\cup[j,j')} \BK{ 1 - 2Y_i(\e^{-1}t) },
\end{align}
in Section~\ref{sect:fluxi},
we establish Proposition~\ref{prop:fluxii} relying on the following exact relations
\begin{align}\label{eq:grd}
	&
	\grd_t(x)
	\in (0, Y_{\I_t(x)-1}(\e^{-1}t)),
	\quad
	\text{ for all } x \text{ such that } x\geq \Xe_{(0)}(t),
\\
	&
	\label{eq:flugtagp}
	\flug_t(x) - \tagpp_t(x) 
	= \e^{1/4}\calDe(\I_t(x),\I_0(x),t) + 2 \e^{1/4} \rho^\e_t(x),
\\
	&
	\label{eq:tagpptagp}
	\tagpp_t(x+\e^b) - \tagp_t(x) = \e^{1/4}\calDe(\I_0(x+\e^b),\ie(x),t).	
\end{align}
Indeed, \eqref{eq:grd} holds since $ \rho^\e_t(x) $ represents
the gap between $ \e^{-1/2}x $ and the first particle to its right,
\eqref{eq:flugtagp} follows by combining \eqref{eq:flug}--\eqref{eq:I} and \eqref{eq:rho},
and  \eqref{eq:tagpptagp} follows by combining and \eqref{eq:zetae} and \eqref{eq:xi}.

The starting point of proving Proposition~\ref{prop:fluxii} is as follows.
We establish part~\ref{enu:fluflug} based on using
$ 
	\Psi_\delta(y,x+\e^b) 
	\approx \ind_{(-\infty,-x-\e^b]}(y) + \ind_{(-\infty,x+\e^b]}(y) 
$,
for $ b\in(0,1/4) $ to ensure that
$ \anglebk{ \Qe_t, \ind_{(-\infty,-x-\e^b]}} \approx 0 $.
As for parts~\ref{enu:flugtagpp} and \ref{enu:tagpptagp},
by shifting each $ x $ by $ \e^{b} $,
we use \eqref{eq:grd} to ensure that $ \e^{1/4}\rho^\e_t(x+\e^b)\approx 0 $,
and by using stationarity, we have $ \calDe(j,j',t) = O(|j-j'|^{1/2}) $.
Consequently,
we reduce showing parts~\ref{enu:flugtagpp} and \ref{enu:tagpptagp} to showing
\begin{align*}
	\e^{1/4} \absBK{ \I_t(x) - \I_0(x) }^{1/2} \approx 0,
	\quad
	\e^{1/4} \absBK{ \I_0(x+\e^b) - \ie(x) }^{1/2} \approx 0.
\end{align*}
The former should hold since, by \eqref{eq:flug}--\eqref{eq:I},
we have $ \I_t(x) - \I_0(x) = \e^{-1/4}(\flug_t(x) - \flug_t(x)) = O(\e^{-1/4}) $,
and we expect the latter to be true since
$ \I_0(x+\e^b)\sim\Pois(2\e^{-1/2}(x+\e^{b})) $
and $ \ie(x) = 2\e^{-1/2}x + O(1) = 2\e^{-1/2}(x+\e^b) + O(\e^{-1/2+b}) $.

The rest of this paper is organized as follows.
Section~\ref{sect:ApriEst} is primarily devoted to the proof of Propositions~\ref{prop:Ito} and \ref{prop:Xotig}.
In Sections~\ref{sect:fluConv} and \ref{sect:fluxi},
we prove Propositions~\ref{prop:fluConvv} and \ref{prop:fluxii}, respectively.

\section{A-priori estimates: Proof of Propositions~\ref{prop:Ito} and \ref{prop:Xotig}}
\label{sect:ApriEst}

Let
$	\Xl_i(t) := X_i(0) + \Be_i(t)$,
$	\Xr_i(t) := \Xl_{i}(t) + \e^{-1/2}t$,
$ \Xl_{(i)}(t) $ and $ \Xr_{(i)}(t) $ be the corresponding ranked processes.
We then have from \eqref{eq:CBP} (for $ \gamma=1 $) that, almost surely, 
for all $ i\in\N $ and $ t\geq 0 $,
%
\begin{align}
	\label{eq:rank}
	\Xl_{i}(t) \leq X_{i}(t) \leq \Xr_{i}(t),
\end{align}
from which it easily follows that
\begin{align}
	\label{eq:rankk}
	\Xl_{(i)}(t) \leq X_{(i)}(t) \leq \Xr_{(i)}(t).
\end{align}
%
%
%
%
Based on \eqref{eq:rank}--\eqref{eq:rankk},
we now establish bounds on the mass of the empirical measure on intervals of the form $ (-\infty,x] $.
\begin{lemma}\label{lem:LebBd}
Fix any $ a>0 $, $q\in[1,\infty)$, $t \in\bbR_+$ and $j\in\N$.
There exists $ C=C(a,q,t)<\infty $ such that, for all $ \e\in (0,(aq)^{-2}] $,
\begin{align}
	&
	\label{eq:Leb:unranked}
	\sum_{i=j}^\infty
	\BVert \sup_{s\in[0,t]} \exp\BK{ -a \Xe_{i}(s) } \BVert_q 
	\leq
	C \e^{-1/2} 
	e^{ - j\e^{1/2}a /4 },
\\
	&
	\label{eq:Leb:ranked}
	\BVert
		\sum_{i=j}^\infty
		\sup_{s\in[0,t]}
	 	\exp\BK{ -a \Xe_{(i)}(s) } 
	 \BVert_q
	\leq
	C \e^{-1/2} 
	e^{ - j\e^{1/2}a /4 }.
\end{align}
\end{lemma}

\begin{proof}
Fix $ t\in\bbR_+ $, $q\in[1,\infty)$, $ a>0 $ and $j_*\in\N$.
Let $\Xlp_i(s) := \Xl_{i+j_*}(s)$
be the $ i $-th (unranked) particle among $ \{\Xl_{j}\}_{j\geq j_*} $.
Let
$ \Fe_i := \sup_{s\in[0,t]} \exp(-a \Xe_i(s))  $,
$ \Fe_{(i)} := \sup_{s\in[0,t]} \exp(-a \Xe_{(i)}(s)) $,
and similarly let $ \Fl_i $, $ \Fl_{(i)} $, $ \Flp_i $ and $ \Flp_{(i)} $
be the corresponding random variables for
$ \Xl_{i} $, $ \Xl_{(i)} $, $ \Xlp_{i} $, $ \Xlp_{(i)} $, respectively.

By \eqref{eq:rank}, $ \Fe_i \leq \Fl_i $, hence
$	
	\sum_{i=j}^\infty \Vertbk{ \Fe_i }_q
	\leq
	\sum_{i=j}^\infty \Vertbk{ \Fl_i }_q.
$
Let $ r := 2^{-1}aq\e^{1/2} $ and $ \barB^\e_i(t) := \sup_{s\in[0,t]} |\Be_i(s)| $.
With $\Xl_i(t) $ defined as in the preceding,
we have
\begin{align}\label{eq:expmom}
	\Ex \BK{\Fl_i}^q 
	\leq
	\BK{ \Ex e^{-2rY_0(0)} }^{i} 
	\Ex\BK{ e^{aq\barB^\e_i(t)} }
	=
	\BK{ 1+r }^{-i} 
	\Ex\BK{ e^{aq\barB^\e_i(t)} }.
\end{align}
Further, by the reflection principle, 
$\Ex[\exp(-aq\barB^\e_i(t))] \leq 2\Ex[\exp(aq\Be_i(t))] = C(a,q,t)$.
Consequently,
\begin{align*}
	\sum_{i=j}^\infty \VertBK{ \Fl_i }_q 
	\leq
	\frac{ (1+r)^{-(j-1)/q} }{ (1+r)^{1/q}-1 } C.
\end{align*}
With $ r \in (0,1] $,
further using the elementary inequalities
$ (1+r)^{1/q} \geq 1+r/q $ and $(1+r)^{-j/q} \leq \exp(-jr/(2q))$,  
we conclude \eqref{eq:Leb:unranked}.

We next show \eqref{eq:Leb:ranked}.
Since, by definition,
$ \Xlp_{(i)}(s) $ is the $ i $-th smallest particles
among $ \{\Xl_{(j)}(s)\}_{j\geq j_* } $,
we have that $ \Xlp_{(i)}(s) \leq \Xl_{(i+j_*)}(s) \leq \Xe_{(i+j_*)}(s) $
and, therefore, $ \Fe_{(i+j_*)} \leq \Flp_{(i)} $.
Summing both sides over $i$,  we further obtain
$
	\sum_{i=0}^\infty \Fe_{(i+j_*)} 
	\leq 
	\sum_{i=0}^\infty \Flp_{(i)}
	=
	\sum_{i=0}^\infty \Flp_{i}
	=
	\sum_{i=j_*}^\infty \Fl_{i}.
$
From this and \eqref{eq:Leb:unranked} we conclude \eqref{eq:Leb:ranked}.
\end{proof}

Based on \eqref{eq:rank}, 
we now establish the continuity of the process $ \Xe_{(i)}(\Cdot) $.
\begin{lemma}\label{lem:locflu}
There exists $ C<\infty $ such that for any $[t_1, t_2] \subset[0,\infty)$, 
$ j\in\N $ and $ \e\in(0,1] $,
\begin{align}\label{eq:locflu}
	\Pro 
	\BK{ \sup_{t\in[t_1,t_2]} \absBK{ \Xe_{(j)}(t) - \Xe_{(j)}(t_1) } \geq \alpha } 
	\leq  
	C \exp\BK{ -\alpha \e^{-1/2} + 2\e^{-1}(t_2-t_1) }.
\end{align}
\end{lemma}

\begin{proof}
It clearly suffices to show that
\begin{align}\label{eq:locflu:}
	\Ex 
	\sqBK{ \exp \BK{ \e^{-1/2} \sup_{t\in[t_1,t_2]} \absBK{ \Xe_{(j)}(t) - \Xe_{(j)}(t_1) } } }
	\leq  
	C \exp\BK{ 2\e^{-1}(t_2-t_1)},
\end{align}
Since $(Y_i(\Cdot))_{i\in\N}$ is at equilibrium,
we have
\begin{align*}
	\big( \Xe_{(i)}(\Cdot+t_{1})-\Xe_{(i)}(t_{1}) \big)_{i\in\N}
	\stackrel{\text{distr.}}{=}
	\big( \Xe_{(i)}(\Cdot)-\Xe_{(i)}(0) \big)_{i\in\N},
\end{align*}
so without lost of generality we assume that $t_1=0$.
Let
\begin{align}
	&
	\label{eq:Ur}
	\Ur(t,i,j) :=
	\sup_{s\in[0,t]}
	\curBK{ \exp\sqBK{ \e^{-1/2} \BK{\Xr_{i}(s) - \Xr_{(j)}(0)) } } },
\\
	&
	\label{eq:Ul}
	\Ul(t,i,j) 
	:=	
	\sup_{s\in[0,t]}
	\curBK{ \exp\sqBK{ -\e^{-1/2} \BK{ \Xl_{i}(t) - \Xl_{(j)}(0)) } } }.
\end{align}
Similar to \eqref{eq:expmom}, we have
\begin{align}
	&
	\label{eq:UrBd}
	\Ex\BK{ \Ur(t,i,j) }
	\leq
	\BK{ \Ex(e^{ -Y_0(0) }) }^{j-i} 
	\Ex\BK{ e^{\e^{-1/2}\barB^\e_i(t)+\e^{-1}t} }
	\leq
	(2/3)^{j-i} C e^{2\e^{-1} t},
	\
	\forall i \leq j,
\\
	&
	\label{eq:UlBd}
	\Ex\BK{ \Ul(t,i,j) }
	\leq
	\BK{ \Ex(e^{ -Y_0(0) }) }^{i-j} 
	\Ex\BK{ e^{\e^{-1/2}\barB^\e_i(t)} }
	\leq
	(2/3)^{i-j} C e^{\e^{-1} t},
	\
	\forall i \geq j.	
\end{align}
By \eqref{eq:rank},
$
	\exp[\e^{-1/2} |\Xe_{(j)}(t) - \Xe_{(j)}(0)|]
	\leq
	\exp[ \e^{-1/2} (\Xr_{(j)}(t) - \Xr_{(j)}(0) ) ]
	+ \exp[ -\e^{-1/2} (\Xl_{(j)}(t) - \Xl_{(j)}(0) ) ].
$
For all $ t\in[0,t_2] $, the last two terms are bounded by 
$ \sum_{i\leq j} \Ur(t_2,i,j) $ and $ \sum_{i\geq j}\Ul(t_2,i,j) $,
respectively.
Combining this with \eqref{eq:UrBd}--\eqref{eq:UlBd},
we conclude \eqref{eq:locflu:}.
\end{proof}

Based on Lemma~\ref{lem:LebBd},
we now establish the following a-priori estimate of the empirical measure.

\begin{lemma}\label{lem:Iebd}
Fix $ T\in\bbR_+ $, $q\in [1,\infty)$ and $ a\in(0,\infty) $.
Let $\Je_j:=[\e^{-1/2}j,\e^{-1/2}(j+1))\cap\bbZ$
and $f_i$, $ i\in\N $,  be $\bbR_+$-valued random variables.
There exits $ C=C(T,q,a)<\infty $ such that for 
all $ t\in[0,T] $ and $ \e\in(0,(aq)^{-2}] $,
\begin{align}\label{eq:Iebd}
	\BVert \sum_{i=0}^\infty f_i e^{-a\Xe_{(i)}(t)} \BVert_q
	\leq
	C \e^{-1/4}
	\sum_{j=0}^\infty e^{-ja/4} 
	\Big( \sum_{i\in\Je_j} \VertBK{f_i}^2_{2q} \Big)^{1/2}.
\end{align}
\end{lemma}

\begin{proof}
For each $j\in\N$, 
by the Cauchy--Schwarz inequality we have
\begin{align*}
	\BVert \sum_{i\in\Je_j} f_i e^{-a\Xe_{(i)}(t)} \BVert_q
	\leq
	\BVert \sum_{i\in\Je_j} e^{-2a\Xe_{(i)}(t)} \BVert^{1/2}_{q}
	\
	\BVert \sum_{i\in\Je_j} (f_i)^2 \BVert^{1/2}_{q}.
\end{align*}
On the r.h.s.,
replacing
$
	\Vertbk{ \sum_{i\in\Je_j} (f_i)^2 }_{q}
$
with
$
	\sum_{i\in\Je_j} \Vertbk{ (f_i)^2 }_{q}
	= \sum_{i\in\Je_j} \Vertbk{ (f_i) }^2_{2q},
$
and
replacing
$ 
	\Vertbk{ \sum_{i\in\Je_j} e^{-2a\Xe_{(i)}(t)} }_{q} 
$
with
$
	\Vertbk{ \sum_{i\geq \e^{-1/2}j} e^{-2a\Xe_{(i)}(t)} }_{q}
$,
which, by \eqref{eq:Leb:ranked},
is bounded by
$
	C \e^{-1/2} \exp( -ja/2),
$
we conclude \eqref{eq:Iebd}.
\end{proof}

Now we establish a decomposition of $ \Bul_t $ into $ \Bulp_t $ and $ \Re_t $ as follows.
As we show latter in \eqref{eq:Rbd}, 
$ \Re_t $ becomes negligible as $ \e\to 0 $, 
so $ \Bul_t \approx \Bulp_t $.

\begin{lemma}\label{lem:Vdec}
Fix $ t\in\bbR_+ $, $ \e\in(0,1] $ and $ \phi\in\Qsp $  
such that $ \frac{d\phi}{dy}\in\Qsp $,
and let
\begin{align}
	&
	\label{eq:W}
	\angleBK{\Bulp_t,\phi} 
	:= 
	\e^{1/4} \sum_{i=0}^\infty 
	\phi\BK{ \Xe_{(i)}(t) } \BK{ 1 - 2 Y_{i}(\e^{-1}t) },
\\ 
	&
	\label{eq:R}
	\angleBK{\Re_t,\phi} 
	:= 
	\e^{-1/4} \sum_{i=0}^\infty 
	\int_{\Xe_{(i)}(t)}^{\Xe_{(i+1)}(t)} \BK{\Xe_{(i+1)}(t) - y}  
	\phi(y) dy.
\end{align}
Then,
\begin{align}\label{eq:Vdec}
	\angleBK{\Bul_t,\phi} =
	\angleBK{\Bulp_t,\phi} - 2\angleBK{\Re_t,\frac{d\phi}{dy}}.
\end{align}
\end{lemma}

\begin{proof}
Since the gaps are at equilibrium,
$ \Xe_{(i)}(t) - \Xe_{(0)}(t) $ is the sum of the i.i.d.\ $\Exp(2\e^{-1/2})$ random variables,
so by the Law of Large Numbers we have $\lim_{k\to\infty} \Xe_{(k)}(t)=\infty$, hence
\begin{align*}
	\angleBK{\Bul_t,\phi}
	=
	\e^{1/4} 
	\sum_{i=0}^\infty 
	\BK{ 
		\phi\BK{ \Xe_{(i)}(t) }
		- 2\e^{-1/2}
		\int_{\Xe_{(i)}(t)}^{\Xe_{(i+1)}(t)} \phi(y) dy
		}.
\end{align*}
With
$	
	\int_{x_1}^{x_2} \phi(y) dy 
	= (x_2-x_1)\phi(x_1) + \int_{x_1}^{x_2}(x_2-y)\phi'(y) dy
$,
we obtain the desired decomposition.
\end{proof}

Based on Lemma~\ref{lem:Iebd},
we next establish bounds on $ \anglebk{\Re_t,\phi} $ and $ \anglebk{\Bulp_t,\phi} $.
We note here that, while these bounds fall short of proving
Proposition~\ref{prop:fluConvv},
they suffice for justifying the use of Ito calculus in Proposition~\ref{prop:Ito}.

Hereafter, when the context is clear, we sometimes
use $\phie_i$, $\Ye_i$ and $ \Xe_{(i)} $, respectively, to denote
$ \phi(\Xe_{(i)}(t)) $, $ Y_i(\e^{-1}t) $ and $ \Xe_{(i)}(t) $.

\begin{lemma}\label{lem:RVbd}
Fix $T\in\bbR_+$, $q\in[1,\infty)$ and $\phi\in\Qsp$ such that $ \frac{d\phi}{dy} \in\Qsp $.
There exists $ C=C(T,q)<\infty $ such that for all 
$ t\in[0,T] $ and $\e\in(0,(2q)^{-2}]$,
\begin{align}
	\label{eq:Rbd}
	&
	\VertBK{ \angleBK{\Re_t,\phi} }_q
	\leq
	 C \e^{1/4} \Qnorm{\phi},
\\
	\label{eq:Wbd}
	&
	\VertBK{ \angleBK{\Bulp_t,\phi} }_q
	\leq 
	C \Qnorm{\scriptstyle\frac{d\phi}{dy}}.
\end{align}
\end{lemma}

\begin{proof}
Fixing $T\in\bbR_+$, $ t\in[0,T] $, $q\in[1,\infty)$, $ \e\in(0,(2q)^{-2}] $ and $\psi\in\Qsp$,
we let $C=C(T,q)<\infty$.
To show \eqref{eq:Rbd},
in \eqref{eq:R}, we use $\Xe_{(i+1)}-y \leq \e^{1/2} Y_{i}$ and
\begin{align*}
	\sup_{y\in[\Xe_{(i)},\Xe_{(i+1)}]} |\phi(y)| \leq \Qnorm{\phi} \exp(-\Xe_{(i)})
\end{align*}
to obtain
$
	| \anglebk{\Re_t,\phi} |
	\leq 
	\e^{3/4} \Qnorm{\phi}
	\sum_{i=0}^\infty
	( Y_{i} )^2 \exp(-\Xe_{(i)}).
$
Combining this with \eqref{eq:Iebd} for $ f_i = (Y_i)^2 $, we arrive at
\begin{align*}
	\VertBK{ \angleBK{\Re_t,\phi} }_q
\leq 
	C \e^{1/2} \Qnorm{\phi}
	\sum_{j=0}^\infty 
	\exp\BK{-j/4}	
	\BK{ \VertBK{ (Y_i)^2 }^2_{2q} |\Je_j| }^{1/2}.
\end{align*}
Further using $ \Vertbk{ (Y_i)^2 }_{2q} = C$
and $|\Je_j| \leq \e^{-1/2}+1$,
we conclude \eqref{eq:Rbd} upon summing $ j $.

Turning to showing \eqref{eq:Wbd},
we assume without lost of generality $ q\in\N\cap[1,\infty) $.
Letting $Z_k:=\sum_{i=0}^k(1-2Y_{i})$,
with $\phi\in\Qsp$,
using summation by parts in \eqref{eq:W},
we obtain
\begin{align}\label{eq:Wbd:W1}
	\angleBK{\Bul_t,\phi} :=
	\e^{1/4}\sum_{i=0}^\infty (\phie_i-\phie_{i+1}) Z_{i}.
\end{align}
To bound this expression, we combine
\begin{align*}
	|\phie_{i+1}-\phie_{i}| 
	\leq \Qnorm{ \frac{d\phi}{dy} } 
	\int_{\Xe_{(i)}}^{\Xe_{(i+1)}} e^{-y} dy
	\leq 
	\Qnorm{ \frac{d\phi}{dy} } 
	\e^{1/2} \Ye_i \exp(-\Xe_{(i)}),
\end{align*}
(where the second inequality is obtained by using $ e^{y} \leq e^{-X_{(i)}} $)
and \eqref{eq:Iebd} for $ f_i = Y_iZ_i $ to obtain
\begin{align}\label{eq:Wbd:W1bd}
	\VertBK{ \angleBK{\Bulp_t,\phi} }_q
	\leq 
	 C
	\e^{1/2}
	\Qnorm{ \frac{d\phi}{dy} }
	\sum_{j=0}^\infty e^{-j/2}
	\BK{ \sum_{i\in \Je_j} \VertBK{ Z_{i} Y_i }^2_{2q} }^{1/2}.	
\end{align}
With $\Vertbk{Y_i}_{4q}=C$ and $\Vertbk{Z_{i}}_{4q} \leq (i+1)^{1/2} C$,
we have $ \Vertbk{  Y_i Z_{i}  }^2_{2q} \leq (i+1)C $.
Plugging this into \eqref{eq:Wbd:W1bd},
we further obtain
\begin{align*}
	\VertBK{ \angleBK{\Bulp_t,\phi} }_q
	\leq
	 C
	\e^{1/2} \Qnorm{ \frac{d\phi}{dy} }
	\sum_{j=0}^\infty \sqBK{ |\Je_j| \e^{-1/2} (j+1) }^{1/2} e^{-j/4}.
\end{align*}
With $|\Je_j| \leq \e^{-1/2} +1$, 
upon summing over $ j $ we conclude \eqref{eq:Wbd}.
\end{proof}

Based on Lemma~\ref{lem:Iebd}, we now establish a bound on $ \Me_{t_0,t}(\psi,j) $, 
as defined as in \eqref{eq:Mek}.
Hereafter we adopt the convention that $ \Me_{t_0,t}(\psi,-1) : =0 $.

\begin{lemma}\label{lem:Mapri}
Let $ \sigma\in[0,\infty] $ be arbitrary stopping time
(with respect to the underlying sigma algebra).
Fix $T\in\bbR_+$ and $ q\in(1,\infty) $.
There exists $ C = C(T,q)<\infty $ such that, for all $\psi\in\Qsp_T$, $ t_0\in[0,T] $,
$ j,j'\geq -1 $ and $ \e\in(0,1] $,
\begin{align}\label{eq:Mapri}
	\VertBK{ 
		\sup_{t\in[t_0,T]} 
		\absBK{\Me_{t_0,t\wedge\sigma}(\psi,j)-\Me_{t_0,t\wedge\sigma}(\psi,j')} 
		}^2_q
	\leq
	C \QTnorm{\psi}^2  \exp\BK{ -(j\wedge j')\e^{1/2}/2 }.
\end{align}
\end{lemma}

\begin{proof}
Fixing such $T$, $q$, $ t_0 $, $ j,j' $, $\e$, $\psi$ and $ \sigma $, we let $C=C(T,q)<\infty$.
We assume without lost of generality $ j>j' $.
Applying Doob's $L^q$-inequality and the \ac{BDG} inequality 
(e.g.\ \cite[Theorem~ \textrm{II}.1.7 and Theorem~ \textrm{IV}.4.1]{revuz99})
to the $ C([t_0,T],\bbR) $-valued 
martingale $M^{\e,*}_{t}:=\Me_{t_0,\Cdot\wedge\sigma}(\psi,j)-\Me_{t_0,\Cdot\wedge\sigma}(\psi,j')$,
we obtain
\begin{align}\label{eq:M:BDG}
\begin{split}
	&
	\BVert 
		\sup_{t\in[t_0,T]} \absBK{M^{\e,*}_{t}} 
	\BVert^2_q
	\leq
	C	
	\VertBK{ 
		\e^{1/2}\int_{t_0}^{T\wedge\sigma} 
		\sum_{i=j'+1}^{j} \BK{\partial_y\psi_s\BK{\Xe_{i}(s)}}^2 ds 
	}_{q/2}
\\
	&
	\quad
	\leq
	C \int_0^T \e^{1/2} 
	\sum_{i=j'+1}^{j} 
	\VertBK{ \BK{ \partial_y\psi_s\BK{\Xe_{i}(s)} }^2 }_{q/2} ds.
\end{split}
\end{align}
In the last expression, 
replacing $ (\partial_y\psi_s(y))^2 $ with $ \QTnorm{\psi}^2 e^{-2y} $
and replacing $ j $ with $ \infty $,
and then applying \eqref{eq:Leb:unranked} for $ a=2 $,
we further obtain the bound $C \QTnorm{\psi}^2 \exp(-j\e^{1/2}/2)$,
thereby concluding \eqref{eq:Mapri}.
\end{proof}

\begin{proof}[Proof of Proposition~\ref{prop:Ito}]
Fix $ \psi\in\Qsp_T $.
The bound \eqref{eq:Mapri} implies that 
$ \{\Me_{t_0,\Cdot}(\psi,j)\}_{j} $ is Cauchy in the complete space
$L^{q}(C([t_0,T],\bbR),\scrB,\Pro)$,
whereby we conclude \eqref{eq:Mcon}.
Further, for all $ q>1 $,
\begin{align}\label{eq:MLqbd}
	\VertBK{ \sup_{t\in[t_0,T]} \absBK{\Me_{t_0,t}(\psi,\infty) } }_q
	\leq
	\lim_{j\to\infty} \VertBK{ \sup_{t\in[t_0,T]} \absBK{\Me_{t_0,t}(\psi,j) } }_q
	\leq
	C(T,q) \QTnorm{\psi},
\end{align}
where the last inequality follows by \eqref{eq:Mapri} for $ j'=-1 $.

To derive \eqref{eq:Ito},
we apply Ito's formula to 
\begin{align*}
	\angleBK{\Qe_{k,s},\psi_s}
:=
	\e^{1/4} \BK{
		\sum_{i=0}^k \psi_t(\Xe_i(s)) - 2 \e^{-1/2} \int_0^\infty \psi_s(y) dy
		}
\end{align*}
to obtain
\begin{align*}
	&
	\left< \Qe_{k,s}, \psi_s \middle> \right|^{s=t}_{s=0}
	=
	\int_0^t 
	\angleBK{ \e^{1/4} \EM_{k,s}, \BK{\partial_s+\frac12 \partial_{yy}} \psi_s } ds
	-
	2 \e^{-1/4} \int_0^t \int_0^\infty \partial_{s} \psi_s(y) dy ds
\\
	&
	\quad
	+ \Me_{0,t}(\psi,k)
	+
	\e^{-1/4} 
	\int_0^t 
	\BK{\partial_y\psi_s} \BK{\Xe_{(0)}(s)}
	\sum_{i=0}^k \ind_\curBK{X_{(i)}(s) = X_{(0)}(s) } ds.
\end{align*}
Clearly, almost surely for all $ s\in[0,T] $,
$ \anglebk{\Qe_{k,s},\phi} \to \anglebk{\Qe_s,\phi} $
and $ \sum_{i=0}^k \ind_\curBK{X_{(i)}(s) = X_{(0)}(s) } \to 1$
as $ k\to\infty $.
As for $ \Me_{0,t}(\psi,k) $, from \eqref{eq:Mcon} (for large enough $ q $)
we deduce that, almost surely for all $ t\in[0,T] $, $ \Me_{0,t}(\psi,k) \to \Me_{0,t}(\psi,\infty) $.
Hence letting $ k\to\infty $ we arrive at
\begin{align}
	\label{eq:Ito:1}
	&
	\left< \Qe_{s}, \psi_s \middle> \right|^{s=t}_{s=0}
	=
	\int_0^t 
	\angleBK{ \e^{1/4} \EM_{s}, \BK{\partial_s+\frac12 \partial_{yy}} \psi_s } ds
	-
	2 \e^{-1/4} \int_0^t \int_0^\infty \partial_{s} \psi_s(y) dy ds
\\
	&
	\label{eq:Ito:2}
	\quad
	+
	\e^{-1/4} 
	\int_0^t 
	\BK{\partial_y\psi_s} \BK{\Xe_{(0)}(s)} ds 	+ \Me_{0,t}(\psi,\infty).
\end{align}
With $ \Ae_t $ and $ \Bul_t $ defined as in \eqref{eq:A}--\eqref{eq:V},
the r.h.s.\ of \eqref{eq:Ito:1} equals
\begin{align}\label{eq:Ito:3}
	\int_0^t \angleBK{ \Bul_t, (\partial_s+2^{-1}\partial_{yy}) \psi_s } ds
	+ \int_0^t \angleBK{\Ae_s,\partial_s\psi_s} ds
	+ \e^{-1/4} \int_0^t \int_{\Xe_{(0)}(s)}^{\infty} \partial_{yy}\psi_s dyds.
\end{align}
The last term in \eqref{eq:Ito:3} cancels the first term in \eqref{eq:Ito:2},
so \eqref{eq:Ito} follows.
\end{proof}

\begin{corollary}
For any $ T\in\bbR_+ $ and $ q\in(1,\infty) $,
there exists $ C=C(T,q)<\infty $ such that
for all $q>1$, $\e\in(0,(2q)^{-2}]$ and $t\in[0,T]$,
\begin{align}\label{eq:Xosech}
	\VertBK{ \int_0^{\Xe_{(0)}(t)} \sech(y) dy }_q
	\leq
	C \e^{1/4}.
\end{align}
\end{corollary}

\begin{proof}
Applying Proposition~\ref{prop:Ito} for $\psi(y):=\sech(y)\in\Qsp_T$,
we obtain
\begin{align*}
	\left<\Ae_s+\Bul_s,\sech \middle> \right|_{s=0}^{s=t} 
=
	2^{-1} \int_0^t 
	\angleBK{ \Bul_s, {\scriptstyle\frac{d^2~}{dy^2}} \sech } ds
	+ \Me_{0,t}(\sech,\infty),
\end{align*}
or equivalently
\begin{align*}
	\angleBK{ \Ae_t, \sech }
=
	\angleBK{ \Bul_0 - \Bul_t,\sech} +
	2^{-1} \int_0^t 
	\angleBK{ \Bul_s,{\scriptstyle\frac{d^2~}{dy^2}} \sech  } ds
	+ \Me_{0,t}(\sech,\infty).
\end{align*}
Recall from \eqref{eq:Vdec} we have
$ 
	\anglebk{\Bul_s,\phi} = \anglebk{\Bulp_s,\phi} - 2\anglebk{\Re_s,\frac{d\phi}{dy}}.
$ 
As $ \psi\in C^\infty(\bbR) $ and $ \frac{d^k}{dy^k}\sech \in\Qsp $ for all $ k\in\N $,
further applying \eqref{eq:Rbd}--\eqref{eq:Wbd} and \eqref{eq:MLqbd},
we conclude \eqref{eq:Xosech}.
\end{proof}

\begin{proof}[Proof of Proposition~\ref{prop:Xotig}]
Fix $ T\in\bbR_+ $, $ b\in[0,1/4) $ and $ q>1 $.
Applying Chebyshev's inequality in \eqref{eq:Xosech}, 
we obtain that, for all $ t\in[0,T] $, $ q> 1 $  and $\e\in(0,(2q)^{-2}]$,
\begin{align}\label{eq:Xomom}
	\Pro \BK{ \absBK{ \Xe_{(0)}(t) } \geq \lambda }
	\leq
	\e^{q/4} C(T,q) \BK{ \int_0^\lambda \sech(y) dy }^{-q}.
\end{align}
Indeed, letting $ \te_k := \e k $, we have
\begin{align}\label{eq:Xotig:union}
\begin{split}
	&
	\curBK{ \taue_b \leq T } 
\\
	&
	\quad
	\subset
	\bigcup_{k\leq \e^{-1}T} 
	\BK{
	\curBK{ |\Xe_{(0)}(\te_k)| \geq \frac{\e^{b}}{2} }
	\cup
	\curBK{ 
		\sup_{t\in[\te_k,\te_{k+1}]} \absBK{\Xe_{(0)}(t)-\Xe_{(0)}(\te_k) } 
		\geq 
		\frac{\e^{b}}{2}
		}
	}.
\end{split}
\end{align}
From \eqref{eq:Xomom} and \eqref{eq:locflu} we deduce
\begin{align}
	&
	\label{eq:Xotig:1}
	\Pro \BK{ 
		\absBK{ \Xe_{(0)}(\te_k) } \geq \e^{b}/2 
	}
	\leq
	C \e^{(1/4-b)q},
\\
	&
	\label{eq:Xotig:2}
	\Pro \BK{ 
		\sup_{t\in[\te_k,\te_{k+1}]} 
		\absBK{ \Xe_{(0)}(t) - \Xe_{(0)}(\te_k) } \geq \e^{b}/2 
	}
	\leq
	C e^{-\e^{b-1/2}/2 }.
\end{align}
In \eqref{eq:Xotig:union} applying the union bound 
using \eqref{eq:Xotig:1}--\eqref{eq:Xotig:2},
we conclude \eqref{eq:taub}.
\end{proof}

Recall $ \EM_t $ is defined as in \eqref{eq:EM}.
We next derive bounds on 
$
	\Pe_t := \e^{1/2} \EM_t.
$
To this end, we let
\begin{align}
	&
	\label{eq:EMz}
	\anglebk{ \EMz_t, \phi } := 
	\anglebk{ \EM_t, \phi(\Cdot+\Xe_{(0)}(t)) },	
\\
	&
	\label{eq:Kb}
	\Se_b(t) := \ind_\curBK{ \sup_{s\in[0,t]} |\Xe_{(0)}(s)| \leq \e^{b} }.
\end{align}

\begin{lemma}\label{lem:Ppapri}
Fix $s,t\in(0,\infty)$, $ x,y'\in\bbR $, $q\in[1,\infty)$, $ b\in [0,1/4) $.
There exists $ C=C(q)<\infty $ such that, for all $ \e\in(0,1] $,
\begin{align}
	&
	\label{eq:pDapri}
	\VertBK{ \Se_b(t) \ \angleBK{ \Pe_{t}, \pN_s(\Cdot-y',x) } }_q
	\leq 
	\BK{ \absBK{\log s}+1 } C,
\\
	&
	\label{eq:p0Dapri}
	\VertBK{ \angleBK{ \Pe_{0}, \pN_s(\Cdot,x) } }_q
	\leq C.
\end{align}
\end{lemma}

\begin{proof}
With $ \pN_s(y,x) := p_s(y-x) + p_s(y+x) $
and $ \Se_b(t) $ decreasing in $ b $,
it clearly suffices to prove, for any fixed $ x'\in\bbR $,
\begin{align}
	&
	\label{eq:Ppapri}
	\VertBK{ \Se_0(t) \ \angleBK{ \Pe_{t}, p_s(\Cdot-x') } }_q
	\leq
	\BK{ \absBK{\log s}+1 } C,
\\
	&
	\label{eq:P0papri}
	\VertBK{ \angleBK{ \Pe_{0}, p_s(\Cdot-x') } }_q
	\leq C.
\end{align}
Since $ p(z) $ decreases in $ |z| $,
we have 
$ 
	p_s(z)
	\leq s^{-1/2} \sum_{j=0}^\infty p(j) \ind_{[j,j+1)}(|z|s^{-1/2}).
$
Using this, we obtain
\begin{align}
	&
	\label{eq:Ppapri:t}
	\Se_0(t) \anglebk{ \Pe_{t}, p_s(\Cdot-x') } 
	=
	\Se_0(t) \e^{1/2} \anglebk{ \EM_{t}, p_s(\Cdot-x') } 
	\leq
	\sum_{j=0}^\infty \Se_0(t) \Fe_j(t,s) p(j),
\\
	&
	\label{eq:Ppapri:0}
	\anglebk{ \Pe_{0}, p_s(\Cdot-x') }
	=
	\e^{1/2} \anglebk{ \EM_{0}, p_s(\Cdot-x') }
	\leq
	\sum_{j=0}^\infty \Ge_j(s) p(j),
\end{align}
where
\begin{align*}
	&
	\Fe_j(t,s) 
	:=
	s^{-1/2} \e^{1/2}
	\angleBK{ \EM_t, \ind_{[j,j+1)} \BK{ \absBK{\Cdot-x'}s^{-1/2} } },
\\
	&
	\Ge_j(s) := 
	s^{-1/2} \e^{1/2} \angleBK{ \EM_0, \ind_{[j,j+1)} \BK{ \absBK{\Cdot-x'}s^{-1/2} } }.
\end{align*}
With $ \EM_0 \sim\PPPp $, we have that $ \Vertbk{\Ge_j}_q \leq C(q) $.
Combining this with \eqref{eq:Ppapri:0}, 
using $ \sum_{j=0}^{\infty} p(j) < \infty $,
we conclude \eqref{eq:P0papri}.
As for \eqref{eq:Ppapri:t}, 
letting
\begin{align*}
	\He_j(t,s) :=
	\sup_{|x''-x'|\leq 1}
	\curBK{
		s^{-1/2} 
		\anglebk{ \EMz_t, \ind_{[j,j+1)} ( |\Cdot-x'|s^{-1/2} ) }
	},
\end{align*}
since $ \EM_t $ and $ \EMz_0 $ differ only by the shift of $ \Xe_0(s) $,
with $ \Se_0(t) $ as in \eqref{eq:Kb},
we have
$
	\Se_0(t) \Fe_j(t,s) \leq \He_j(t,s).
$
With $ \EMz_t\sim\PPPp $,
\eqref{eq:Ppapri} now follows in a way similar to \eqref{eq:P0papri}.
The only difference is the maximum over $ \{ x'': |x''-x'|\leq 1\} $,
which results in the extra $ |\log s| $ factor.
\end{proof}

\section{Proof of Proposition~\ref{prop:fluConvv}}\label{sect:fluConv}

\subsection{Proof of part~\ref{enu:conato}}
Fixing $ b\in(0,1/4) $, $ b'\in(1/8,1/4) \cap [b,\infty) $ and $ T\in\bbR_+ $,
we show
\begin{align}\label{eq:ato:claim}
	\lim_{(\e,\delta)\to (0,0)}
	\Se_{b'}(T)
	\BK{
		\sup_{t\in[0,T]} \sup_{x\in\bbR_+}
		\absBK{ \atoo_{t}(x)  }
	}  
	= 0.
\end{align}
The desired result $ \atoo_\Cdot(\Cdot) \Rightarrow 0 $
then follows since $ \Se_{b'}(T) \to_\text{P} 1 $ (by Proposition~\ref{prop:Xotig}).

Turning to proving \eqref{eq:ato:claim},
fixing $ t\in[0,T] $,
by \eqref{eq:A} and \eqref{eq:ato} we have
\begin{align*}
	\Se_{b'}(T) \absBK{ \atoo_t(x) } 
	\leq 
	2\e^{-1/4} \Se_{b'}(T)
	\int_0^t \int_0^{\Xe_{(0)}(s)}
	\absBK{ \partial_s\Psi_{t+\delta-s}(y,x+\e^b) }
	dyds.
\end{align*}
Since here $ \sup_{s\in[0,T]}\{|\Xe_{(i)}(s)|\} \leq \e^{b'} $, 
we may integrate over $ \int_{-\delta}^{T+1} \int_{-\e^{b'}}^{\e^{b'}} $ instead.
After exchanging the order of integrations,
we integrate over $ s\in(-\delta,T+1) $
using the readily verified identity
$ |\partial_s \Psi_s(y,x+\e^{b})| = -\sign(y) \partial_s\Psi(y,x+\e^{b}) $
to obtain
\begin{align}\label{eq:atocon:}
	\Se_{b'}(T) \absBK{ \atoo_{t}(x) }
	\leq
	2\e^{-1/4} \int_{-\e^{b'}}^{\e^{b'}} 
	\absBK{ \Psi_{T+1+\delta}(y,x+\e^b) - \Psi_{0}(y,x+\e^b) } dy.
\end{align}
Let $ f(y) :=  \Psi_{T+1+\delta}(y,x+\e^b) - 1 $.
Since $ \Psi_0(y,x+\e^b) =1 $, for all $ x\geq 0 $ and $ |y| \leq \e^{b'} \leq \e^{b} $,
we have
$ 
	| \Psi_{T+1+\delta}(y,x+\e^b) - \Psi_{0}(y,x+\e^b) | 
	= |f(y)|
$. 
Further, since $ f(0)=0 $ and $ f'(y) = -\pN_{T+\delta+1}(y,x+\eta) $, 
we further deduce
$ 
	|f(y)| \leq
	C |y| (T+1+\delta)^{-1/2}
	\leq C |y|.
$
Plugging this into \eqref{eq:atocon:},
we obtain
$
	\Se_{b'}(T) | \atoo_{t}(x) |
	\leq
	C \e^{-1/4+2b'}, 
$
thereby, with $ b'>1/8 $, concluding \eqref{eq:ato:claim}.

\subsection{Proof of part~\ref{enu:conic}}
Recall $ \Psibfe_t(y,x) $ and $ \pNbfe_t(y,x) $ are defined as in \eqref{eq:PsipNbfe}.
By Lemma~\ref{lem:Vdec}, we have $ \icQ_t(x) = \ic_t(x) - 2 \rd_t(x) $,
for
\begin{align}
	&
	\label{eq:ic}
	\ic_t(x) := \e^{1/4} \sum_{i=0}^\infty 
	\BK{1-2Y_i(0)} \Psibfe_{t}(\Xe_{(i)}(0),x),
\\
	&
	\label{eq:rd}
	\rd_t(x) := \e^{-1/4} \sum_{i=0}^\infty \int_{\Xe_{(i)}(0)}^{\Xe_{(i+1)}(0)} 
	\BK{ \Xe_{(i+1)}(0) - y } \pNbfe_{t}(y,x) dy.
\end{align}

We first show that $ \rd_\Cdot(\Cdot) \Rightarrow 0 $,
or more explicitly,
\begin{align}\label{eq:rdBd}
	\Ex\BK{ \sup_{t\in[0,T]} \sup_{x\in[0,L]} \absBK{ \rd_{t}(x) } }
	\leq
	C \e^{1/4}|\log\e|,
\end{align}
for some $ C=C(T,L)<\infty $ and for all $ \e\in(0,1/4] $ and $ \delta,\eta \in(0,1] $.

\begin{proof}[Proof of \eqref{eq:rdBd}]
Fixing $T,L\geq 0$, we let $ C=C(T,L) $.
To bound $ \rd_t(x) $, in \eqref{eq:rd} we replace $ (\Xe_{(i+1)}(0)-y) $ with $ \e^{1/2} Y_i(0) $,
and then divide the sum into the sums over $ i\leq \e^{-1} $ and over $ i>\e^{-1} $.
For the former replacing each $ Y_i(0) $ (with $ i\leq\e^{-1} $)
by $ \overline{Y}^\e := \sup_{i\leq \e^{-1}} Y_i(0) $,
we obtain
\begin{align}
	&
	\label{eq:rdBd:R12}
	\sup_{t\in[0,T]} \sup_{x\in[0,L]} \absBK{ \rd_t(x) } 
	\leq
	 \Re_1 + \Re_2,
\\ 
	&
	\notag
	\Re_1 :=
	\e^{1/4} \overline{Y}^\e \int_{\Xe_{(0)}(0)}^{\Xe_{(\ceilbk{\e^{-1}})}(0)} \pNbfe_{t}(y,x) dy
	\leq 2 \e^{1/4} \overline{Y}^\e,
\\
	&
	\label{eq:rdBd:R2}
	\Re_2 :=
	\e^{1/4} \sum_{i>\e^{-1}} Y_i
	\sup_{t\in[0,T]} \sup_{x\in[0,L]} 
	\int_{\Xe_{(i)}(0)}^{\infty} \pNbfe_{t}(y,x) dy.
\end{align}
With $ \{Y_i(0)\} \sim \bigotimes_{i\in\N} \Exp(2) $,
we have $ \Ex (\Re_1) \leq C \e^{1/4}|\log\e| $.
As for $ \Re_2 $, 
from \eqref{eq:Psi} we have 
\begin{align}\label{eq:PhiQsp}
 	0 \leq \Psibfe_t(x,y) \leq C(T,L) (e^{-y}\wedge 1), \quad \forall 
 	t\in[0,T], \ x\in[0,L],\  y\in\bbR_+.
\end{align}
Plugging this into \eqref{eq:rdBd:R2}, we obtain
$	
	\Re_2 \leq
	C \e^{1/4} \sum_{i>\e^{-1}} Y_i \exp(-\Xe_{(i)}(0)).
$
Further applying \eqref{eq:Iebd} for $ f_i = Y_i $,
we conclude
\begin{align*}
	\Ex(R_2) 
	\leq C
	\sum_{j=0}^\infty e^{-j/4}
	\BK{ \sum_{i\in\Je_j} \ind_\curBK{i>\e^{-1}} \VertBK{Y_i}^2_2 }^{1/2}
	\leq
	\e^{-1/4}C
	\exp\BK{-\e^{-1/2}/4}.
\end{align*}
Combining the preceding bounds on $ \Ex(R_1) $ and $ \Ex(R_2) $ with \eqref{eq:rdBd:R12},
we conclude \eqref{eq:rdBd}.
\end{proof}

With \eqref{eq:ato:claim}, it then suffices to show:
\begin{lemma}\label{lem:ictig}
We have that $\{\ic_{\Cdot}(\Cdot)\}_{\bfe} \subset C(\bbR^2_+,\bbR) $
and the processes are tight in $ C(\bbR^2_+,\bbR) $. 
\end{lemma}
\begin{lemma}\label{lem:icConv}
As $ \bfe\to\bfz $, $\{\ic_{\Cdot}(\Cdot)\}_{\bfe}$
converges in finite dimensional distribution 
to a centered Gaussian process $ \icc_\Cdot(\Cdot) $
with the covariance \eqref{eq:icCov}.
\end{lemma}
\noindent
We prove Lemma~\ref{lem:ictig} (as well as Lemma~\ref{lem:Mtig}) 
by applying the following special form of the Kolmogorov--Chentsov criterion of tightness
(see \cite[Corollary 14.9]{kallenberg02}).

\begin{lemma}[Kolmogorov--Chentsov]\label{lem:KC}
A given collection of $ C(\bbR^2_+,\bbR) $-valued processes
$ \{K^{\bfe}_\Cdot(\Cdot)\}_{\bfe} $ is tight if,
for some $ \alpha\in(0,1] $, and for all $ q\in(1,\infty) $, $ T, L \in\bbR_+ $,
there exists $ C=C(T,L,\alpha,q) \geq 0 $ such that
\begin{align}
	&
	\label{eq:KCbd}
	\Vertbk{K^{\bfe}_0(0)}_q \leq C,
\\
	&
	\label{eq:KCsp}
	\Vertbk{ K^{\bfe}_{t}(x) - K^{\bfe}_{t}(x') }_q \leq |x-x'|^{\alpha/2}C,
\\
	&
	\label{eq:KCti}
	\Vertbk{ K^{\bfe}_{t}(x) - K^{\bfe}_{t'}(x) }_q \leq |t-t'|^{\alpha/4}C,	
\end{align}
for all $ t,t'\in[0,T] $, $ x,x'\in[0,L] $,
$ \e $, $ \delta $ and $ \eta $ sufficiently small.
\end{lemma}

\begin{proof}[Proof of Lemma~\ref{lem:ictig}]
For each $ i\in\N $,
$ (t,x) \mapsto (1-2Y_i(0)) \Psibfe_{t}(\Xe_{(i)}(0),x) $
is continuous.
The series \eqref{eq:ic} defining $ \ic_{\Cdot}(\Cdot) $
converges absolutely,
hence $ \ic_{\Cdot}(\Cdot) \in C(\bbR^2_+,\bbR) $.

Fixing $ T,L\in\bbR_+ $, $ q\in(1,\infty) $, $ x,x'\in[0,L] $ and $ t<t'\in[0,T] $
letting $ C=C(T,L,q) <\infty $,
we next show \eqref{eq:KCbd}--\eqref{eq:KCti}
for $ K^{\bfe}_t(x) = \ic_{t}(x) $ and $ \alpha=1 $.
Consider the discrete time martingale
\begin{align}\label{eq:Mbfe}
	k \longmapsto
	\mbfe_k(t,x) := \e^{1/4} \sum_{i=0}^k (1-2Y_i(0))\Psibfe_t(\Xe_{(i)}(0),x).
\end{align}
With $ \ic_t(x) = \mbfe_{\infty}(t,x) $,
showing \eqref{eq:KCbd}--\eqref{eq:KCti} amounts to
bounding the quadratic variation of  $ \mbfe_\Cdot(t,x) $, 
which we do by using $ \EM_0\sim\PPPp $.

Let $ \anglebk{\Pek_0,f} := \e^{1/2} \sum_{i=0}^k f(\Xe_{(i)}(0)) $
be the $ k $-th approximation of $ \Pe_t $.
The martingale $ \mbfe_k(t,x) $ has quadratic variation 
$ \anglebk{ \Pek_0, \Psibfe_t(\Cdot,x))^2 } $.
Consequently, by the \ac{BDG} inequality and Fatou's lemma, 
letting $ k\to\infty $ we have
\begin{align}
	&
	\label{eq:ictig:ho}
	\VertBK{ \ic_0(0) }^2_{q} \leq C \VertBK{ \angleBK{\Pe_0,(\Psibfe_{0}(\Cdot,0))^2} }_{q/2},
\\
	&
	\label{eq:ictig:space}
	\VertBK{ \ic_{t}(x) - \ic_{t}(x') }^2_{q} \leq C 
	\VertBK{ \angleBK{\Pe_0,(\Psibfe_{t}(\Cdot,x)-\Psibfe_{t}(\Cdot,x'))^2} }_{q/2},
\\
	&
	\label{eq:ictig:time}
	\VertBK{ \ic_{t}(x) - \ic_{t'}(x) }^2_{q} \leq C
	\VertBK{ \angleBK{\Pe_0,(\Psibfe_{t}(\Cdot,x)-\Psibfe_{t'}(\Cdot,x))^2} }_{q/2}.
\end{align}

The estimate \eqref{eq:KCbd}
follows by applying $ \Psibfe_{0}(y,0) \leq C e^{-y} $
(by \eqref{eq:PhiQsp}) to \eqref{eq:ictig:ho}
and then using $ \Vertbk{\anglebk{\Pe_0, \exp(-2\Cdot)}}_{q/2} \leq C $
(by \eqref{eq:Leb:ranked} for $ j=0 $).
To show \eqref{eq:KCsp},
since $ 0 \leq \Psibfe_t(y,x) \leq 2 $, we have
\begin{align}\label{eq:ictig:xhol}
	\BK{ \Psi_{t+\delta}(y,x) - \Psibfe_{t}(y,x') }^2
	\leq
	2
	\int_x^{x'} \absBK{ \partial_z \Psibfe_{t}(z,x) } dz
	=
	2 \int_x^{x'} \pNbfe_{t}(y,z) dz.
\end{align}
Using this in \eqref{eq:ictig:space},
we bound the r.h.s.\ of \eqref{eq:ictig:space} by
$ 
	C \int_x^{x'} \Vertbk{\anglebk{\Pe_0,\pNbfe_{t}(\Cdot,z)}}_{q/2} dz.
$
This, by \eqref{eq:p0Dapri}, is bounded by $ C |x-x'| $,
whereby we conclude \eqref{eq:KCsp}.
Turning to showing \eqref{eq:KCti},
letting $ \Psitilbfe_{t,t'}(y) := \Psibfe_{t}(y,x) - \Psibfe_{t'}(y,x) $,
similar to \eqref{eq:ictig:xhol} we have
\begin{align}
	&
	\notag
	\BK{ \Psitilbfe_{t,t'}(y) }^2
	\leq
	2
	\int_{t}^{t'} \absBK{ \partial_s \Psibfe_s(y,x) } ds
\\
	&
	\label{eq:ictig:thol}
	\quad
	=
	\int_{t}^{t'} s^{-1}
	\absBK{ (y+x+\eta)\pbfe_s(y+x) + (y-x-\eta)\pbfe_s(y-x) } ds.
\end{align}
However, due to the $ s^{-1} $ singularity, the argument 
for proving \eqref{eq:KCsp} does not apply.
To circumvent this problem,
letting
$ 
	g(y) := \ind_\curBK{|x+\eta-y|\leq |t'-t|^{1/2}} 
	\vee \ind_\curBK{|x+\eta+y|\leq |t'-t|^{1/2}},
$
we bound 
$
	\Fe_1 : = 
	\anglebk{ \Pe_0, (1-g) (\Psitilbfe_{t,t'})^2 }
$
and
$
	\Fe_2 : = 
	\anglebk{ \Pe_0, g (\Psitilbfe_{t,t'})^2 }
$
separately.
For $ \Fe_1 $, in \eqref{eq:ictig:thol} 
using $ |s^{-1}zp_s(z)| \leq C|z|^{-1}p_{2s}(z) $
and $ |x+\eta\pm y| \geq |t'-t|^{1/2} $,
we obtain
\begin{align*}
	\BK{ \Psitilbfe_{t,t'}(y) }^2 \BK{ 1-g(y) }
	\leq 
	(t-t')^{1/2} C \int_{t}^{t'} \pNbfe_{2s}(y,x) ds.
\end{align*}
Given this inequality,
we now conclude $ \Vertbk{\Fe_1}_{q/2} \leq C|t-t'| $
by the same argument following \eqref{eq:ictig:xhol}.
As for $ \Fe_2 $, using $ |\Psitilbfe_{t,t'}(y)| \leq 2 $,
we obtain
\begin{align*}
	\Fe_2 \leq 4\e^{1/2} 
	\angleBK{ \Pe_0, \ind_{[x+\eta-(t'-t)^{1/2},x+\eta+(t'-t)^{1/2}]} + \ind_{[-x-\eta-(t'-t)^{1/2},-x-\eta+(t'-t)^{1/2}]} }.
\end{align*}
Combining this with $ \Pe_0 \sim \PPPp $, 
we conclude $ \Vertbk{\Fe_2}_{q/2} \leq C |t'-t|^{1/2} $.
\end{proof}

Next we prove Lemma~\ref{lem:icConv}
using the martingale Central Limit Theorem of \cite[Theorem 2]{brown71}, 
which we state here in the form convenient for our purpose. 
\begin{lemma}[martingale Central Limit Theorem]\label{lem:mgCLT}
Suppose that for any fixed $ \bfe\in(0,1]^3 $, 
$ (\Nbfe_i,\scrFbfe_i) $, $ i= -1, 0, 1,\ldots, \ne $,
is a discrete time $ L^2 $-martingale, 
starting at $ \Nbfe_{-1} = 0 $, 
with the corresponding martingale differences 
$ \Dbfe_i := \Nbfe_{i+1} - \Nbfe_{i} $
and predictable compensator 
$ \anglebk{\Nbfe }_i := \sum_{i'=0}^i \Ex[(\Dbfe_{i'})^2|\scrFbfe_{i'-1}] $.
If, for some $ \sigma_*\in\bbR_+ $, as $ \bfe\to\bfz $,
\begin{align}
	&
	\label{eq:CLTLind}
	\sum_{i=0}^{\ne}
	\Ex \BK{ \absBK{\Dbfe_i}^3 } 
	\longrightarrow
	0,
\\
	&
	\label{eq:CLTqv}
	\angleBK{\Nbfe}_{\ne} \xrightarrow[\text{P}]{} \sigma^2_*,
\end{align}
then $ \Nbfe_{\ne} \Rightarrow \calN(0,\sigma_*) $, 
the mean zero Gaussian with variance $ \sigma_*^2 $.
\end{lemma}

\begin{remark}
Although the proof of \cite[Theorem 2]{brown71} is for a single martingale,
the same proof applies for a collection of of martingales $ \{(\Nbfe_i,\scrFbfe_i)\}_{\bfe} $ 
as we consider here.
In particular, the truncation argument of \cite{brown71}
applies equally wells here, by letting $ \tau_{\bfe,L}:= \inf\curBK{i:\angleBK{\Nbfe}_i>L} $,
whereby $ \Pro(\Nbfe_i=\Nbfe_{i\wedge\tau_{\bfe,L}},\forall i) \to 1 $
as $ L\to\infty $, uniformly in $ \bfe $.
\end{remark}

\begin{proof}[Proof of Lemma~\ref{lem:icConv}]
Let 
$
	\Gamma_{t,t'}(x,x') := 2\int_0^\infty \Psi_t(y,x) \Psi_{t'}(y,x') dx.
$
Fixing arbitrary $ t_1,\ldots, t_l $ and $ x_1,\ldots,x_l\in\bbR_+ $,
we let $ C=C(t_1,\ldots,t_l,x_1,\ldots,x_l)<\infty $ and
\begin{align*}
	\bfic := (\ic_{t_1}(x_1),\ldots,\ic_{t_l}(x_l)) \in \bbR^l.
\end{align*}
Our goal is to show $ \bfic \Rightarrow \calN(0,\Sigma) $,
where $ \Sigma := (\Gamma_{t_j,t_{j'}}(x_j,x_{j'}))_{j,j'=1}^l $.
Equivalently, fixing arbitrary $ \bfv=(v_i)\in\bbR^l $
and letting
$
	\sigma_* := [ \sum_{j,j'=1}^l v_j v_{j'} \Gamma_{t_j,t_{j'}}(x_j,x_{j'}) ]^{1/2},
$
we show 
\begin{align*}
	\bfv\Cdot\bfic = \sum_{j=1}^l v_j \mbfe_\infty(t_j,x_j) \Rightarrow \calN(0,\sigma_*),
\end{align*}
where $ \mbfe_k(t,x) $ is defined as in \eqref{eq:Mbfe}.
To this end, 
letting $ \ne:= \ceilbk{\e^{-1}} $,
we consider the martingale 
\begin{align}\label{eq:Nbfe}
	\Nbfe_i := 
	\sum_{j=1}^l v_j\mbfe_{k_i}(t_j,x_j),
	\quad
	k_i := i \ind_\curBK{i<\ne} + \infty \ind_\curBK{i=\ne}.
\end{align}
It then suffices to verify 
\begin{enumerate*}[label=\itshape\roman*\upshape)]
\item \eqref{eq:CLTLind};
and \item \eqref{eq:CLTqv}.
\end{enumerate*}

(\textit{i})
Let
$ 
	F^{\bfe}_i 
	:= 
	\sum_{j=1}^{l} v_j \Psibfe_{t_j}(\Xe_{(i)}(0),x_j).
$
With $ \Nbfe_i $ defined as in \eqref{eq:Nbfe}, we have
\begin{align}\label{eq:Lypbd}
	\sum_{i=0}^{\ne} \Ex(|\Dbfe_i|^3) 
	\leq
	\e^{3/4} 
	\Bigg[ 
		\sum_{i\leq \e^{-1} } \Ex\BK{ \absBK{1-2Y_i(0)}^3 (\Fbfe_{i})^3 }
		+
		\Ex\Big( \sum_{i\geq\e^{-1}} (1-Y_i(0)) \Fbfe_i \Big)^3
	\Bigg].
\end{align}
We now show that the r.h.s.\ tends to zero based on the a-priori estimates \eqref{eq:Leb:ranked} and \eqref{eq:Iebd}.
From \eqref{eq:PhiQsp} we obtain $ |\Fbfe_i| \leq C \exp(-\Xe_{(i)}(0)) $.
Using this in \eqref{eq:Lypbd},
we bound the r.h.s. by $ \gbfe_1 + \gbfe_2 $, where
\begin{align}
	&
	\label{eq:icCon:Lyo}
	\gbfe_1 :=
	\e^{3/4} \Ex\Big( \sum_{i\leq\e^{-1}} \absBK{1-2Y_i(0)}^3 e^{-3\Xe_{(i)}(0)} \Big),
\\
	&
	\notag
	\gbfe_2 :=
	\e^{3/4} \BVert \sum_{i\geq\e^{-1}} (1-Y_i(0)) e^{-\Xe_{(i)}(0)} \BVert^3_3.
\end{align}
In \eqref{eq:icCon:Lyo}, replacing each $ |1-2Y_i(0)|^3 $ with
$ \overline{Y}^{\e,*} := \sup_{i\leq \e^{-1}}\{|1-2Y_i(0)|^3\} $,
we obtain
\begin{align*}
	\gbfe_1 
	\leq
	\e^{3/4}
	\BVert \overline{Y}^{\e,*} \BVert_2
	\
	\BVert \sum_{i=0}^\infty e^{-3\Xe_{(0)}(i)} \BVert_2.
\end{align*}
With $ \{Y_i(0)\}\sim\bigotimes_i \Exp(2) $,
we have $ \Vertbk{ \overline{Y}^{\e,*} }_2 \leq C(|\log\e|+1) $,
and by \eqref{eq:Leb:ranked} for $ j=0 $,
we have $ \Vertbk{ \sum_{i=0}^\infty e^{-3\Xe_{(i)}(0)} }_2 \leq C\e^{-1/2} $,
whereby we conclude $ \gbfe_1 \to 0 $.
As for $ \gbfe_2 $,
applying \eqref{eq:Iebd} for $ f_i = |1-Y_i(0)| $,
we obtain
$
	\gbfe_2 \leq \e^{3/4} [\e^{-1/2} \exp(-\e^{-1/2}/C)]^3 \to 0.
$

(\textit{ii})
With $ \Nbfe_i $ defined as in \eqref{eq:Nbfe},
we have
$
	\angleBK{\Nbfe}_{\ne}
	= 
	\sum_{j,j'}^l v_j v_{j'} \Gamma^{\bfe}_{t_{j},t_{j'}}(x_j,x_{j'}),
$
where
\begin{align}\label{eq:icCon:qv}
	\Gamma^{\bfe}_{t,t'}(x,x')
	= 
	\e^{1/2} \sum_{i=0}^\infty \Psibfe_{t}(\Xe_{(i)}(0),x) \Psibfe_{t'}(\Xe_{(i)}(0),x').
\end{align}
In \eqref{eq:icCon:qv}, if we replace each $ \Xe_{(i)}(0) $
by $ \Ex(\Xe_{(i)}(0)) = \e^{1/2}2^{-1}i := \xe_i $,
we obtain the expression 
\begin{align}\label{eq:Gammabfe}
	\Gamma^{\bfe,*}_{t,t'}(x,x')
	:=
	\e^{1/2} \sum_{i=0}^\infty \Psibfe_{t}(\xe_i,x) \Psibfe_{t'}(\xe_i,x').
\end{align}
This, with $ \xe_{i+1}-\xe_{i} = 2^{-1}\e^{1/2} $,
is a Riemann sum approximation of $ \Gamma_{t,t'}(x,x') $.
In particular, by using the continuity of $ (y,\bfe)\mapsto\Psibfe_t(y,x) $,
it is not hard to show that $ \Gamma^{\bfe,*}_{t,t'}(x,x') \to \Gamma_{t,t'}(x,x') $.
Consequently, showing \eqref{eq:CLTqv} is reduced to showing 
$ \Gamma^{\bfe}_{t,t'}(x,x') - \Gamma^{\bfe,*}_{t,t'}(x,x') \to_\text{P} 0 $,
which is in turn implied by
\begin{align}\label{eq:icCon:claim}
	\e^{1/2} \sum_{i=0}^\infty
	\Ex\absBK{
		\Psibfe_{t}(\Xe_{(i)}(0),x) \Psibfe_{t'}(\Xe_{(i)}(0),x')
		- \Psibfe_{t}(\xe_i,x) \Psibfe_{t'}(\xe_i,x')
		}
	\to
	0.
\end{align}

We now prove \eqref{eq:icCon:claim} by using the continuity of $ y\mapsto \Psi_{t}(y,x) $
and the control on $ |\Xe_{(i)}(0)-\xe_i| = |\Xe_{(i)}(0)-\Ex(\Xe_{(i)}(0))| $.
For any $ L>0 $, 
we divide the expression in \eqref{eq:icCon:claim} into $ \GbfeL_1 + \GbfeL_2 $, for
\begin{align}
	&
	\GbfeL_1 :=
	\e^{1/2}
	\sum_{i>L\e^{-1/2}}
	\absBK{
		\Psibfe_{t}(\Xe_{(i)}(0),x) \Psibfe_{t'}(\Xe_{(i)}(0),x')
		- \Psibfe_{t}(\xe_i,x) \Psibfe_{t'}(\xe_i,x')
		},
\\
	&
	\label{eq:icCon:G2}
	\GbfeL_2 :=
	\e^{1/2}
	\sum_{i\leq L\e^{-1/2}}
	\absBK{
		\Psibfe_{t}(\Xe_{(i)}(0),x) \Psibfe_{t'}(\Xe_{(i)}(0),x')
		- \Psibfe_{t}(\xe_i,x) \Psibfe_{t'}(\xe_i,x')
		}.
\end{align}
By \eqref{eq:PhiQsp} and \eqref{eq:Leb:ranked},
for the tail term $ \GbfeL_1 $ we have $ \Ex(\GbfeL_1) \leq C\exp(-L/C) $.
With this,
it then suffices to show 
\begin{align}\label{eq:icCon:G2goal}
	\lim_{\bfe\to\bfz} \Ex(\GbfeL_2) = 0,
	\text{ for any fixed } L>0,
\end{align}
(since we can then further take $ L\to\infty $ after taking $ \bfe\to\bfz $).
To this end, \emph{fixing} arbitrary $ L>0 $, 
we let  $ \de := \sup_{i\le \e^{-1/2}L} |\Xe_{(i)}(0)-\xe_i| $.
With $ \{\Xe_{(i)}(0)\} \sim \PPPp $ we have $ \Pro(|\de|> \e^{1/8}) \to 0 $.
By telescoping,
in \eqref{eq:icCon:G2}, for each $ i $, we bound the corresponding term by 
$
	|\Psibfe_{t}(\Xe_{(i)}(0),x) - \Psibfe_{t}(\xe_i,x) | \Psibfe_{t'}(\xe_i,x')
	+
	\Psibfe_{t}(\xe_i,x) |\Psibfe_{t'}(\Xe_{(i)}(0),x') - \Psibfe_{t'}(\xe_i,x') |.
$
Further using $ |\Psibfe_{t}(x,y)| \leq 2 $
and $ |\Psibfe_s(a,x)-\Psibfe_s(b,x)| = \int_a^b \pNbfe_s(z,x) dz $,
we obtain
\begin{align}\label{eq:icCon:De}
	\GbfeL_2
	\leq 
	C \int_{-\de}^{\de} \anglebk{\Pe_0, \pNbfe_{t}(\Cdot+z,x) +\pNbfe_{t'}(\Cdot+z,x')} dz.
\end{align}
Now consider the cases $ \de\leq \e^{1/8} $ and $ \de> \e^{1/8} $ separately.
For the former combining \eqref{eq:icCon:De} and \eqref{eq:P0papri},
we obtain $ \Ex( \Gbfe_2\ind_\curBK{\de\leq\e^{1/8}} ) \leq C\e^{1/8} \to 0 $.
For the latter using $ \GbfeL_2 \leq C(L) $ (since $ |\Psi_\Cdot(\Cdot)|\leq 2 $),
we conclude
$
	\Ex( \Gbfe_2\ind_\curBK{\De>\e^{1/8}} ) \leq C(L)\Pro(\de>\e^{1/8})\to 0.
$
Therefore \eqref{eq:icCon:G2goal} follows.
\end{proof}

\subsection{Proof of part~\ref{enu:conmg}}
Recall $ \Se_b(\Cdot) $ is defined as in \eqref{eq:Kb}. 
Letting $ \tloc := t\wedge\taue_{1/8} $ and
\begin{align}\label{eq:mgloc}
\begin{split}
	&
	\mgloc_{t,t'}(x) 
	:= 
	\Me_{\tloc,\tploc}
	\BK{ \pNbfe_{t-\Cdot}(\Cdot,x), \infty }
\\
	&
	\quad
	=
	\e^{1/4}
	\sum_{i=0}^\infty
	\int_{t}^{t'}
	\Se_{1/8}(s) 
	\pN_{t'+\delta-s} \BK{ \Xe_i(s),x} d\Be_i(s),
\end{split}
\end{align}
we recall from Proposition~\ref{prop:Xotig}, that for any $ T\in\bbR_+ $,
$ 
	\lim_{\bfe\to\bfz}
	\Pro(\mg_{t}(x)=\mgloc_{0,t}(x), \forall t\in[0,T],x\in\bbR_+) = 1 
$,
so without lost of generality we replace $ \mg_t(x) $ with $ \mgloc_{t}(x) := \mgloc_{0,t}(x) $.

\begin{lemma}\label{lem:Mtig}
The collection of processes $\{\mgloc_\Cdot(\Cdot)\}_{\bfe}\subset C(\bbR_+^2,\bbR)$ 
is tight in $ C(\bbR_+^2,\bbR) $.
\end{lemma}

\begin{proof}
The process $ (t,x)\mapsto \mgloc_{t}(x) $,
as the uniform limit (as $ k\to\infty $) of the continuous martingale
$ \Me_{0,\tloc}(\psi,k) $ for $ \psi_s(y) := \pNbfe_{t-s}(y,x) $,
is continuous.

Fixing $ T,L\in\bbR_+ $, $ q\in(1,\infty) $ and $ \alpha\in(0,1) $,
hereafter we let $ C=C(T,L,q,\alpha)<\infty $.
For this fixed $ \alpha $,
we next verify the conditions \eqref{eq:KCbd}--\eqref{eq:KCti} for $ K^{\bfe}_t(x) = \mgloc_{t}(x) $.
The first condition \eqref{eq:KCbd} follows trivially since $ \mgloc_0(0)=0 $.
As for \eqref{eq:KCsp}--\eqref{eq:KCti},
fixing $t< t'\in[0,T]$, $x< x'\in[0,L]$,
our goal is to bound the moments of
$ N^{\bfe}_1 := \mgloc_{t}(x') - \mgloc_{t}(x)$
and $ \mgloc_{t'}(x) - \mgloc_{t}(x) = N^{\bfe}_2 + N^{\bfe}_3 $,
where
$ 
	N^{\bfe}_2 := 
	\Me_{\tloc}(\pNbfe_{t'-\Cdot}(\Cdot,x)) 
	- \Me_{\tloc}(\pN_{t-\Cdot}(\Cdot,x)) 
$
and 
$ 
	N^{\bfe}_3 := 
	\Me_{\tloc,\tploc}(\pNbfe_{t'-\Cdot}(\Cdot,x)).
$
To this end, we control the quadratic variation
\begin{align}
	&
	\label{eq:Mtig:V1}	
	V^{\bfe}_1 :=
	\int_0^{t} 
	\Se_{1/8}(s) 
	\angleBK{
		\Pe_{s},
		(\pNbfe_{t-s}(\Cdot,x')-\pN_{t-s}(\Cdot,x))^2
		}
	ds,
\\
	&
	\label{eq:Mtig:V2}
	V^{\bfe}_2 :=
	\int_0^{t}
	\Se_{1/8}(s)
	\angleBK{
		\Pe_{s},
		(\pNbfe_{t'-s}(\Cdot,x)-\pN_{t-s}(\Cdot,x))^2
		}
	ds,
\\
	&
	\label{eq:Mtig:V3}
	V^{\bfe}_3 :=
	\int_{t}^{t'}
	\Se_{1/8}(s) 
	\angleBK{\Pe_{s}, (\pNbfe_{t'-s}(\Cdot,x'))^2} ds,
\end{align}
of the martingales $ N^{\bfe}_j $, $ j=1,2,3 $, respectively.
By using
\begin{align}
	&
	\label{eq:pNx}
	\absBK{ \pNbfe_{t}(y,x) - \pNbfe_{t}(y',x') }
	\leq 
	C t^{-(\alpha+1)/2} ( |x-x'|^\alpha + |y-y'|^\alpha ),
\\
	&
	\label{eq:pNt}
	\absBK{ \pNbfe_{t}(y,x) - \pNbfe_{t'}(y,x) }
	\leq 
	C t^{-(\alpha+1)/2} (t'-t)^{\alpha/2},
\\
	&
	\notag
	|\pNbfe_{s'}(y,x)| \leq C (s')^{-1/2}
\end{align}
(where \eqref{eq:pNx} and \eqref{eq:pNt} follow from the 
$ \alpha $-H\"{o}lder continuity of $ \exp(-z^2/2) $ and $ z\exp(z^2/2) $, respectively),
we obtain 
\begin{align*}
	&
	\BK{ \pNbfe_{t-s}(\Cdot,x')-\pNbfe_{t-s}(\Cdot,x) }^2
	\leq
	C |x-x'|^{\alpha} \BK{ \pNbfe_{t-s}(\Cdot,x) + \pNbfe_{t-s}(\Cdot,x') },
\\
	&
	\BK{ \pNbfe_{t'-s}(\Cdot,x)-\pNbfe_{t-s}(\Cdot,x) }^2
	\leq
	C|t-t'|^{\alpha/2} \BK{ \pNbfe_{t'-s}(\Cdot,x') + \pNbfe_{t-s}(\Cdot,x') },
\\
	&
	\BK{ \pNbfe_{t'-s}(\Cdot,x') }^2
	\leq 
	C\BK{ t'-s }^{-1/2} \pNbfe_{t'-s}(\Cdot,x').
\end{align*}
Plugging this in \eqref{eq:Mtig:V1}--\eqref{eq:Mtig:V3},
and using \eqref{eq:p0Dapri},
we further obtain
\begin{align}
	&
	\label{eq:Mtig:Vbd1}
	\VertBK{V^{\bfe}_1}_q
	\leq 
	C|x-x'|^{\alpha} 
	\int_0^{t} (t'-s)^{-(1+\alpha)/2} 
	\BK{ |\log(t-s)|+1 } ds,
\\
	&
	\VertBK{V^{\bfe}_2}_q 
	\leq
	C|t'-t|^{\alpha/2} 
	\int_0^{t} (t'-s)^{-(1+\alpha)/2} \BK{ |\log(t'-s)|+1 } ds,
\\
	&
	\label{eq:Mtig:Vbd3}
	\VertBK{V^{\bfe}_3}_q 
	\leq
	C \int_{t}^{t'} (t'-s)^{-1/2} (|\log(t'-s)|+1) ds,
\end{align}
respectively.
By the Burkholder--Davis--Gundy inequality, we have
$ \Vertbk{N^{\bfe}_j}_q \leq C ( \Vertbk{V^{\bfe}_j}_{q/2} )^{1/2} $.
Combining this with \eqref{eq:Mtig:Vbd1}--\eqref{eq:Mtig:Vbd3},
we thus conclude \eqref{eq:KCsp}--\eqref{eq:KCti} for $ K^{\bfe}_t(x) = \mgloc_{t}(x) $.
\end{proof}

\begin{lemma}\label{lem:mgcon}
As $ \bfe\to\bfz $, 
$\{\mgloc_{\Cdot}(\Cdot)\}_{\bfe}$
converges in finite dimensional distribution 
to a centered Gaussian process $ \mgg_\Cdot(\Cdot) $
with the covariance \eqref{eq:mgCov}.
\end{lemma}

\begin{proof}
Fixing arbitrary $ t_1,\ldots,t_m, x_1,\ldots,x_m \in\bbR_+ $,
we let $ C=C(t_1,\ldots,t_m,x_1,\ldots,x_m)<\infty $,
$ \bfv := (v_1,\ldots,v_m)\in\bbR^m $ and $ \iota := \sqrt{-1} $.
Consider the characteristic functions
$
	\cha_{\bfe}(\bfv)
	:= \Ex[ \exp ( \iota \sum_{j=1}^m v_j \mgloc_{t_{j}}(x_j) ) ]
$
and
$
	\cha(\bfe)
	:= \Ex[ \exp ( \iota \sum_{j=1}^m v_j \mgg_{t_{j}}(x_j) ) ]
$
of $ (\mgloc_{t_j}(x_j))_{j=1}^m $
and $ (\mgg_{t_j}(x_j))_{j=1}^m $, respectively.
By L\'{e}vy's continuity theorem,
it suffices to show $ \cha_{\bfe} \to \cha $.
Letting
$
	\qN_{s,t,t'}(y,x,x')
	:=
	\pN_{t-s}(y,x) \pN_{t'-s}(y,x')
$
and
$
	\calV_{t,t'}(x,x')
	:=
	2 \int_0^{t\wedge t' } \int_0^\infty 
	\qN_{s,t,t'}(y,x,x') dy ds
$,
we recall that
\begin{align*}
	\cha(\bfv)
	= 
	\exp
	\Big(  -2^{-1} \sum_{j,j'=1}^m v_j v_{j'} \calV_{t_j,t_{j'}}(x_j,x_{j'}) \Big).
\end{align*}
As for $ \cha_{\bfe} $,
since $ \mgloc_t(x) $ is a continuous martingale of
quadratic variation
\begin{align}
	&
	\notag
	\calV^{\bfe}_{t,t'}(x,x')
	:= \int_0^{t\wedge t'} \Se_{1/8}(s) \angleBK{ \Pe_s, \qNbfe_{s,t,t'}(\Cdot,x,x') } ds
\\
	&
	\label{eq:mgcon:calV}
	\quad
	=
	\int_0^{t\wedge t'} \Se_{1/8}(s) 
	\BK{ \e^{1/2} \sum_{i=0}^\infty \qNbfe_{s,t,t'}(\Xe_{(i)}(s),x,x') } ds,
\end{align}
where $ \qNbfe_{t,t'}(y,x,x') := \pNbfe_t(y,x) \pNbfe_t(y,x') $,
we have
\begin{align*}
	\cha_{\bfe}(\bfv)
	=
	\Ex \Big[ \exp\Big( -2^{-1}\sum_{j,j'=1}^m v_{j}v_{j'} \calV^{\bfe}_{t_j,t_{j'}}(x_j,x_{j'}) \Big) \Big].
\end{align*}
Given these expressions of $ \cha $ and $ \cha_{\bfe} $, 
by the bounded convergence theorem, it suffices to show that
\begin{align}\label{eq:mgcon:red}
	\calV^{\bfe}_{t,t'}(x,x') \to_\text{P} \calV_{t,t'}(x,x'),
\end{align}
for all $ t,t',x,x'\geq 0$.
Similar to \eqref{eq:icCon:qv}--\eqref{eq:Gammabfe},
in \eqref{eq:mgcon:calV},
if we replace $ \Xe_{(i)}(s) $ with $ \xe_i = \Ex(\Xe_{(0)}(0)) $,
the resulting expression $ \calV^{\bfe,*}_{t,t'}(x,x') $
represents a Riemann sum approximation of $ \calV_{t,t'}(x,x') $.
The only difference here is the extra factor of $ \Se_{1/8}(s) $,
which satisfies $ \Pro(\Se_{1/8}(s) =1,\forall s\in[0,T])\to 1 $
(by Proposition~\ref{prop:Xotig}).
Hence, in the same way $ \Gamma^{\bfe,*}_{t,t'}(x,x')\to\Gamma_{t,t'}(x,x') $,
we have $ \calV^{\bfe,*}_{t,t'}(x,x') \to_{\text{P}} \calV_{t,t'}(x,x') $,
thereby reducing showing \eqref{eq:mgcon:red} to showing
\begin{align}\label{eq:mgcon:goal}
	\int_0^{t\wedge t'}
	\Ex\BK{ \Se_{1/8}(s) \rbfe(s) } ds \longrightarrow 0,
\end{align}
for
$
	\rbfe(s)
	:= 
	\e^{1/2} 
	\sum_{i=0}^\infty 
	| \qNbfe_{s,t,t'}(\Xe_{(i)}(s),x,x') - \qNbfe_{s,t,t'}(\xe_i,x,x') |.
$

We now prove \eqref{eq:mgcon:goal}
by using the continuity of $ y\mapsto\qNbfe_{s,t,t'}(y,x,x') $
and the control on $ \Se_{1/8}(s)| \Xe_{(i)}(s) - \xe_i | $,
similar to the proof of \eqref{eq:icCon:claim}.
Expressing $ \Xe_{(i)}(s) $ as $ \Xe_{(i)}(s) - \Xe_{(0)}(s) + \Xe_{(0)}(s) $,
with $ \calDe(j,j',t) $ defined as in \eqref{eq:calDX}, we have
\begin{align}\label{eq:mgcon:de}
	\Se_{1/8}(s) \absBK{ \Xe_{(i)}(s) - \xe_i }
	\leq
	2^{-1} \e^{1/2} \calDe(0,i,s) + \e^{1/8} := d^{\e}_i(s).
\end{align}
Since the gaps are at equilibrium, from \eqref{eq:calDY} we deduce
\begin{align}\label{eq:calDe:mom}
	\VertBK{\De(j,j',t)}_n \leq 
	C(n) |j-j'|^{1/2}.
\end{align}
Using this for $ (j,j')=(0,i) $, with $ i = 2\e^{-1/2}\xe_i $, we obtain
\begin{align}\label{eq:mgcon:dbd}
	\Vertbk{d^{\e}_i(s)}_n \leq C(n) \e^{1/8} \sqBK{ 1+(\xe_{i})^{1/2} },
	\quad \forall n\in \bbZ.
\end{align}
Next, by telescoping, we bound $ \rbfe(s) $ by $ \Fbfe_{1}(s) + \Fbfe_{2}(s) $, where
\begin{align}
	&
	\label{eq:mgcon:F1}
	\Fbfe_{1}(s) :=
	\e^{1/2} \sum_{i=0}^\infty
	\absBK{ 
		\pNbfe_{t-s}(\Xe_{(i)}(s),x) - \pNbfe_{t-s}(\xe_i,x) 
	}
	\pNbfe_{t'-s}(\Xe_{(i)}(s),x') ,
\\
	&
	\label{eq:mgcon:F2}
	\Fbfe_{2}(s) := 
	\e^{1/2} \sum_{i=0}^\infty
	\pNbfe_{t-s}(\xe_i,x) 
	\absBK{ \pNbfe_{t'-s}(\Xe_{(i)}(s),x') - \pN_{t'-s}(\xe_i,x') }.
\end{align}
In \eqref{eq:mgcon:F1}--\eqref{eq:mgcon:F2},
using \eqref{eq:pNx} for $ \alpha=1/2 $ and using \eqref{eq:mgcon:de},
we obtain
\begin{align}
	&
	\label{eq:mgcon:h1}
	\Se_{1/8}(s) \Fbfe_{1}(s) 
	\leq 
	C (t-s)^{-3/4} 
	\e^{1/2} \sum_{i=0}^\infty
	\BK{ d^{\e}_i(s) }^{1/2} \Se_{1/8}(s) \pNbfe_{t'-s}(\Xe_{(i)}(s),x'),
\\
	&
	\label{eq:mgcon:h2}
	\Se_{1/8}(s) \Fbfe_{2}(s) 
	\leq 
	C (t'-s)^{-3/4}
	\e^{1/2} \sum_{i=0}^\infty
	\BK{ d^{\e}_i(s) }^{1/2} \pNbfe_{t-s}(\xe_i,x).
\end{align}
Plugging \eqref{eq:mgcon:dbd} in \eqref{eq:mgcon:h2}, we obtain
\begin{align}\label{eq:mgCon:F2}
	\Ex\BK{ \Se_{1/8}(s) \Fbfe_{2}(s) }
	\leq 
	C (t'-s)^{-3/4} \e^{1/16}.
\end{align}
As for $ \Fbfe_{1}(s) $, fixing $ q\in(1,2) $,
in \eqref{eq:mgcon:h1},
for each $ i $, multiplying and dividing by the factor $ \exp(-\Xe_{(i)}(s)) $,
we apply H\"{o}lder's inequality
(with respect to $ \Ex[\sum_i(\Cdot)] $)
to obtain $ \Ex(\Fbfe_1(s)) \leq (t-s)^{-3/4} \fbfe_{11}(s) \fbfe_{12}(s) $,
where
\begin{align*}
	&
	\fbfe_{11}(s) :=
	\sqBK{ \Ex\BK{ \e^{1/2} \sum_{i=0}^\infty (d^{\e}_i(s))^{q'/2} e^{-q'\Xe_{(i)}(s)} } }^{1/q'},
\\
	&
	\fbfe_{12}(s) :=
	\sqBK{ \Ex \BK{
		\Se_{b}(s)
		\angleBK{ \Pe_s, \exp(q\Cdot) \BK{ \pNbfe_{t'-s}(\Cdot,x') }^q }
		} }^{1/q},
\end{align*}
and $ 1/q' + 1/q =1 $.
Combining \eqref{eq:mgcon:dbd} and \eqref{eq:Iebd} for $ f_i = (d^{\e}_i(s))^{q'} $,
we obtain $ \fbfe_{11}(s) \leq C\e^{1/16} $.
As for $ \fbfe_{12}(s) $,
with $ e^{qy} p_{\sigma}(y-z) = C(\sigma,z) p_{\sigma}(y-q\sigma-z) $,
we have 
\begin{align*}
	&
	e^{qy} \BK{ \pNbfe_{t'-s}(y,x') }^q
	=
	\sqBK{ \pNbfe_{t'-s}(y,x') }^{q-1} e^{qy} \pNbfe_{t'-s}(y,x')
\\
	&
	\quad
	\leq
	C(L,T,q) (t'-s)^{-(q-1)/2} \pNbfe_{t'-s}(y-q(t'+\delta-s),x').
\end{align*}
Using this and \eqref{eq:pDapri} for $ y'= q(t+\delta-s) $,
we obtain
$  \fbfe_{12} \leq C(t-s)^{-(q-1)/2q} (|\log(t'-s)|+1)^{1/q} $.
Consequently,
\begin{align}\label{eq:mgCon:F1}
	\Ex\BK{ \Fbfe_{1}(s) }
	\leq 
	C \e^{1/16} (t-s)^{-3/4-(q-1)/2q} \BK{|\log(t'-s)|+1}.
\end{align}
With $ (q-1)/2q <1/4 $, from \eqref{eq:mgCon:F2}--\eqref{eq:mgCon:F1}
we conclude \eqref{eq:mgcon:goal}.
\end{proof}

\section{Proof of Proposition~\ref{prop:fluxii}}
\label{sect:fluxi}

Recall $ \te_k := \e^{-1}k $.
We first establish the following estimates on the continuity in $ t $ 
of $ \flug_t(x) $, $ \tagpp_t(x) $ and $ \tagp_t(x) $.

\begin{lemma}
For any fixed $ T,L\in\bbR_+ $,
\begin{align}
	&
	\label{eq:locxi}
	F_{\tagp}(T,L) := 
	\sup\curBK{
	\absBK{ \tagp_t(x) - \tagp_{t_k}(x) }:
		k\leq T\e^{-1}, \ t\in[t_k,t_{k+1}], \ x\in[0,L]
		}
	\xrightarrow[\text{P}]{}
	0,
\\
	&
	\label{eq:loczeta}
	F_{\tagpp}(T,L) := 
	\sup\curBK{ 
		\absBK{ \tagpp_t(x) - \tagpp_{t_k}(x) }:
		k\leq T\e^{-1}, \ t\in[t_k,t_{k+1}], \ x\in[0,L]
		}
	\xrightarrow[\text{P}]{}
	0,
\\
	&
	\label{eq:locG}
	F_{\flug}(T,L) := 
	\sup\curBK{ 
		\absBK{ \flug_t(x) - \flug_{t_k}(x) }:
		k\leq T\e^{-1}, \ t\in[t_k,t_{k+1}], \ x\in[0,L]
		}
	\xrightarrow[\text{P}]{}
	0.
\end{align}
\end{lemma}

\begin{proof}
We say that events $ \{\calA^\e\} $ happen at \ac{SPR} if, 
for each $ q\geq 1 $, $ \Pro((\calA^\e)^c) \e^{-q} $ is uniformly bounded.
By \eqref{eq:xi}, 
$
	\tagp_t(x) - \tagp_{t_k}(x)
	= 2 \e^{1/4} ( X_{(\ie(x))}(\e^{-1}t) - X_{(\ie(x))}(\e^{-1}t_k) ).
$
Fixing arbitrary $ a>0 $,
from \eqref{eq:locflu} we deduce that
\begin{align*}
	\sup_{t\in[t_k,t_{k+1}]} 
	\curBK{ \e^{1/4} |X_{(\ie(x))}(\e^{-1}t) - X_{(\ie(x))}(\e^{-1}t_k)| } 
	\leq a,
	\quad
	\text{ at \ac{SPR}}. 
\end{align*}
By taking the union bound over $ k \leq T\e^{-1} $
and over $ \ie(x) \in \bbZ\cap[0,L\e^{-1/2}+1] $,
which is a union size $ C\e^{-3/2} $,
we conclude that $ F_{\tagp}(T,L) \leq a $ at \ac{SPR}.
As $ a>0 $ is arbitrary, we obtain \eqref{eq:locxi}.

As for \eqref{eq:locxi},
by \eqref{eq:zetae} we have 
$
	\tagpp_t(x) - \tagpp_{t_k}(x)
	= 2 \e^{1/4} ( X_{(\I_0(x))}(\e^{-1}t) - X_{(\I_0(x))}(\e^{-1}t_k) )
$.
Further,
with $ \{\Xe_{(i)}(0)\}\sim\PPPp $, we have
\begin{align}\label{eq:I0bd}
	\I_0(x) \leq (4L+1)\e^{-1/2} \quad \text{ as \ac{SPR},}	
\end{align}
so \eqref{eq:loczeta} follows by the same argument for \eqref{eq:locxi}.

Letting
$
	F_{\flug}(k,L) := 
	\sup\{ 
		| \flug_{t}(x) - \flug_{t_k}(x)| :
		 t\in[t_k,t_{k+1}], \ x\in[0,L]
		\},
$
we next show \eqref{eq:locG} by showing, for each $ k $ and fixed $ a>0 $, 
$ F_{\flug}(k,L)\leq a $ at \ac{SPR}.
By stationarity,
$ 
	|\flug_t(x)-\flug_{0}(x)| \stackrel{\text{distr.}}{=} 
	|\flug_{t+t_k}(x+\Xe_{(0)}(t_k))-\flug_{t_k}(x+\Xe_{(0)}(t_k))|
$
and by Proposition~\ref{prop:Xotig} $ |\Xe_{(0)}(t_k)|\leq 1 $ at \ac{SPR}.
Hence it suffices to show 
$
	F^{*}_{\flug}(L)
	\to_\text{P} 0,
$
where
$
	F^{*}_{\flug}(L)
	:=
	\sup\{ 
		| \flug_{t}(x) - \flug_{0}(x)| : t\in[0,\e], \ x\in[-1,L+1]
		\}	
	\to_\text{P} 0 
$.
With $ \flug_t(x) $ defined as in \eqref{eq:flug}, 
we have that
$ 	
	\flug_t(x) - \flug_{0}(x)
	= \e^{1/4} \Ge(t,x),
$
where
\begin{align}
	&
	\notag
	\Ge(t,x) = 
	\anglebk{ \EM_{t}, \ind_{ (-\infty,x]} } -
	\anglebk{ \EM_{0}, \ind_{ (-\infty,x]} }
\\
	&
	\quad
	\label{eq:locGflux}
	=
	\sum_{i\geq\I_{0}(x)} \ind_\curBK{\Xe_{i}(t) \leq x} - \sum_{i<\I_{0}(x)} \ind_\curBK{\Xe_{i}(t) >x}
\end{align}
is the net flux of particles across $ x $ within $ [0,t] $.
Let
\begin{align}
	&
	\label{eq:locG:H}
	\He(j) := 
	\sum_{i\geq j }  \Hl(i,j)
	+ \sum_{i<j} \Hr(i,j),
\\
	&
	\notag
	\Hl(i,j) := \ind_\curBK{ \inf_{t\in[0,\e]}\limits \Xl_{i}(t) \leq \Xl_{(j)}(0) },
	\quad
	\Hr(i,j) := \ind_\curBK{ \sup_{t\in[0,\e]}\limits \Xr_{i}(t) >\Xr_{(j-1)}(0) }.
\end{align}
In \eqref{eq:locGflux}, using \eqref{eq:rank},
$ \Xl_{(\I_{0}(x))}(0) >x $ and $ \Xe_{(\I_{0}(x)-1)}(0) \leq x $,
we then obtain that\\
$ 
	\sup_{t\in[0,\e]} 
	\{ |\Ge(t,x)| \}
	\leq 
	\e^{1/4}\He(\I_{0}(L)) 
$.
Combining this with \eqref{eq:I0bd}, we now arrive at
\begin{align}\label{eq:locG:k0}
	F^{*}_{\flug}(L)
	\leq 
	\sup_{ j\in\scrK(L) } \curBK{ \e^{1/4} \He(j) }
	\quad
	\text{ at \ac{SPR},}
\end{align}
where
$
	\scrK(L) := 
	\{ 
		j:
		|j| \leq 4(L+1)\e^{-1/2}
		\}.
$
Recall $ \Ur(t,i,j) $ and $ \Ul(t,i,j) $ are defined as in \eqref{eq:Ur}--\eqref{eq:Ul}.
Fixing any $ q\geq 1 $, 
in \eqref{eq:locG:H} taking the $ q $-th norm of both sides, we obtain
\begin{align*}
	& 
	\VertBK{ \He(j) }_q 
	\leq 
	\sum_{i\geq j }  \VertBK{ \Hl(i,j) }_q
	+ \sum_{i<j}  \VertBK{ \Hr(i,j) }_q
\\
	&
	\quad
	\leq
	\sum_{i\geq j }  \BK{ \Ex{ \Ul(\e,i,j) } }^{1/q}
	+ \sum_{i<j}  \BK{ \Ex{ \Ur(\e,i,j) } }^{1/q}.	
\end{align*}
Further using \eqref{eq:UrBd}--\eqref{eq:UlBd} in the last expression,
we conclude $ \VertBK{ \He(j) }_q \leq C(q) $.
With $ q\geq 1 $ being arbitrary and $ |\scrK(L)| \leq C\e^{-1/2} $, 
we have $ \sup_{ j\in\scrK(L) } \{\e^{1/4}\He(j)\} \leq a $ at \ac{SPR},
for arbitrary fixed $ a>0 $.
Combining this with \eqref{eq:locG:k0}, we thus complete the proof.
\end{proof}

Recall $ \calDe(j,j',t) $ is defined as in \eqref{eq:calDX}.
Let
\begin{align}\label{eq:muset}
	\scrI_{\mu'}(T,L) := 
	\curBK{
		(j,j',k)\in \N^3
		: j,j' \leq 4(L+1)\e^{-1/2}, |j-j'|\leq \e^{-\mu'},
		k \leq T\e^{-1}
	}.
\end{align}

\begin{lemma}\label{lem:calD}
For each fixed $ T,L\in\bbR_+ $ and $ \mu'\in(0,1/2) $,
\begin{align}\label{eq:calD}
	\sup_{(j,j',k)\in\scrI_{\mu'}(T,L)}
	\e^{1/4} \absBK{ \calDe(j,j',t_k) } 
	\xrightarrow[\text{P}]{}
	0.
\end{align}
\end{lemma}

\begin{proof}
Fix such $ L,T,\mu' $ and any $ a>0 $.
By \eqref{eq:calDe:mom} we have
$
	\Pro
	\BK{ \e^{1/4}| \calDe(j,j',t_k) | \geq  a } 
	\leq C(n,a) (|j-j'| \e^{1/2} )^n.
$
Form this, 
fixing $ n > 2/(1/2-\mu') $,
using the union bound, we obtain
\begin{align*}
	\Pro
	\BK{ 
		\bigcup_{(j,j',k)\in\scrI_{\mu'}(T,L)}
		\curBK{  \e^{1/4}\absBK{ \calDe(j,j',t_k) } \geq a } 
	} 
	\leq 
	C(n,T,L,a) \e^{n(1/2-\mu')} \absBK{ \scrI_{\mu'}(T,L) }.	
\end{align*}
With $ | \scrI_{\mu'}(T,L) | \leq C(T,L)\e^{-2} $ and $ n(1/2-\mu')-2 >0 $,
the r.h.s.\ tends to zero as $ \e\to 0 $.
Since $ a>0 $ is arbitrary, we conclude \eqref{eq:calD}.
\end{proof}

Hereafter we say events $ \{\calA^\e\}_\e $ occur with
with an \ac{OP} if $ \Pro(\calA^\e) \to 1 $ as $ \e\to 0 $.

\begin{proof}[Proof of Proposition~\ref{prop:fluxii}\ref{enu:fluflug}]

Fixing any $ L, T \in\bbR_+ $, $ a\in(1/2,\infty) $ and $ b\in(0,1/4) $,
our goal is to show
\begin{align}\label{eq:fluxi1}
	\sup_{x\in[\e^{b},L]} \sup_{t\in[0,T]} 
	\absBK{ \flux_t(x) - \flug_t(x) }
	\xrightarrow[\text{P}]{} 0.
\end{align}
Let 
\begin{align}
	&
	\notag
	\fe(y,x) := \Psi_{\e^a}(y,x)-\ind_{(-\infty,x]}(y)
\\
	&
	\label{eq:fluxi1:fe}
	\quad
	= 1- \Phi_{\e^a}(y+x) + \ind_{(x,\infty)}(y) - \Phi_{\e^a}(y-x).
\end{align}
Recall from \eqref{eq:flu} and \eqref{eq:flug} that
$ 
	\flux_t(x) - \flug_t(x) 
	= \anglebk{ \Qe_t, \fe(\Cdot,x) } = \Fe_1(t,x) + \fe_2(x),
$
where
\begin{align}
	&
	\label{eq:fluxi1:f1}
	\Fe_1(t,x) := 
	\e^{1/4}\angleBK{ \EM_t, \fe(\Cdot,x) },
\\
	&
	\label{eq:fluxi1:f2}
	\fe_2(x) := 
	-2 \e^{-1/4} \int_{0}^\infty \fe(y,x) dy.
\end{align}
From \eqref{eq:fluxi1:fe} we deduce that, for $ (x,y)\in\bbR_+^2 $,
\begin{align}\label{eq:fluxi1:}
	|\fe(y,x)| 
	\leq 
	\Phi_{\e^{a}}(-y-x) + \Phi_{\e^{a}}(-|y-x|).
\end{align} 
In particular,
$ \int_0^\infty |\fe(y,x)| dy \leq C \int_0^\infty \Phi_{\e^{-a}}(z) dz = C \e^{a/2} $.
Using this in \eqref{eq:fluxi1:f2},
with $ a>1/2 $, we conclude that $ \sup_{x\geq 0} \{|\fe_2(x)|\} \to 0 $.
It then suffices to show 
\begin{align}\label{eq:flux1:F1}
	\sup \curBK{ |\Fe_1(t,x)|: t\in[0,T],x\in[\e^{b},L] }
	\xrightarrow[\text{P}]{}
	0.
\end{align}

We now show \eqref{eq:flux1:F1}
by using \eqref{eq:fluxi1:} and $ \EMz_t\sim\PPPp $.
Recall on $ (y,x)\in\bbR_+^2 $ we have
$ \Phi_{\e^{a}}(-y-x) \leq C \exp[ -2(x+y)/\e^{a/2} ] $
and
$ \Phi_{\e^{a}}(-|y-x|) \leq C \exp[-2|y-x|/\e^{a/2} ] $.
Combining this with \eqref{eq:fluxi1:},
we obtain
\begin{align*}
	|\fe(y,x)| \ind_\curBK{x\geq \e^{b}}
	\leq 
	C e^{ -2\e^{b-a/2} -2y\e^{-a/2} } + C e^{-2|y-x|\e^{-a/2} }.
\end{align*}
This decays fast expect when $ |y-x| $ is small,
so fixing $ a'\in(1/2,a) $, we further deduce
\begin{align*}
	|\fe(y,x)| \ind_\curBK{x\geq \e^{b}}
	\leq 
	 C \BK{ e^{ -2\e^{b-a/2} } e^{-y} + e^{-\e^{(a'-a)/2}} e^{-(y-x)} } 
	 +
	 C \ind_\curBK{ |y-x|<\e^{a'/2} }.
\end{align*}
Plugging this in \eqref{eq:fluxi1:f1},
we arrive at $ \Fe_1(t,x) \leq  C\Fe_{11}(t,x) + C\Fe_{12}(t,x) $,
where
\begin{align*}
	&
	\Fe_{11}(t)
	:=
	\e^{1/4} \BK{ e^{-\e^{b-a/2}} + e^{-\e^{(a'-a)/2} } }
	\anglebk{ \EM_t, \exp\BK{-\Cdot} },
\\
	&
	\Fe_{12}(t,x) 
	:= \e^{1/4}\anglebk{ \EM_t, \ind_{ (x-\e^{a'/2},x+\e^{a'/2}] }(\Cdot) }.
\end{align*}
For $ \Fe_{11}(t) $, with $ a>a'\vee (2b) $, applying \eqref{eq:Leb:ranked}  
we conclude that $ \sup_{t\in[0,T]} \{\Fe_{11}(t)\} \to 0$ in $ L^1 $ and hence in probability.
Turning to bounding $ \Fe_{12}(t,x) $,
we let 
$ 	
	N(t,x) := 
	\e^{1/4} \anglebk{ \EMz_{t}, \ind_{ (x-\e^{a'/2},x+\e^{a'/2}] } }
	\sim
	\Pois(4\e^{(a'-1)/2}).
$
Since  $ \EM_t $ and $ \EMz_t $ differ only by the shift of $ \Xe_{(0)}(t) $,
which by Proposition~\ref{prop:Xotig} is at most $ 1 $ with an \ac{OP},
we have
\begin{align*}
	\sup_{x\in[0,L]} \Fe_{12}(t,x) \leq \sup_{ x\in[-1,L+1] } \curBK{ \e^{1/4} N(t,x) }
	\leq 
	2 \sup_{ |i| \leq (L+2)\e^{-a'/2} } \curBK{ \e^{1/4} N(t,2i\e^{a'/2}) },
\end{align*}
for all $ t\in[0,T] $ with an \ac{OP}.
Further, with $ a'>1/2 $,
from the large deviations bound of $ \Pois(4\e^{(a'-1)/2}) $,
we deduce that
\begin{align*}
	\sup_{k\leq T\e^{-1}} \sup_{ |i| \leq (L+2)\e^{-a'/2} } \curBK{ \e^{1/4} N(t_k,2i\e^{a'/2}) }
	\xrightarrow[\text{P}]{}
	0,
\end{align*}
thereby concluding 
\begin{align}\label{eq:flug:pre}
	\sup_{k\leq T\e^{-1}} \sup_{x\in[0,L]} \{\Fe_{12}(t_k,x) \} \to_\text{P} 0.
\end{align}
Now, since (by \eqref{eq:flug}) 
$ \Fe_{12}(t,x) = \flug_t(x+\e^{a'/2}) - \flug_t(x-\e^{a'/2}) -2 \e^{a'/2-1/4} $,
combining \eqref{eq:flug:pre} and \eqref{eq:locG}, 
we conclude $ \sup_{t\in[0,T]} \sup_{x\in[0,L]} \{\Fe_{12}(t,x) \} \to_\text{P} 0 $.
\end{proof}
\begin{proof}[Proof of Proposition~\ref{prop:fluxii}\ref{enu:flugtagpp}]
Fixing $ L, T \geq 0 $ and $ b \in(0,1/4) $,
by \eqref{eq:loczeta}--\eqref{eq:locG} and  \eqref{eq:flugtagp}, it suffices to show
\begin{align*}
	\sup_{k \leq T\e^{-1}} \sup_{x\in[\e^{b},L]} 
	\absBK{ \e^{1/4}\calDe(\I_{t_k}(x),\I_0(x), t_k) - 2\e^{1/4} \rho^\e_{t_k}(x) }
	\xrightarrow[\text{P}]{} 0,
\end{align*}
as $ \e\to 0 $.
Letting
\begin{align*}
	&
	\Ge_1 :=	
	\sup_{k\leq T\e^{-1}} \sup_{x\in[\e^b,L]} \curBK{ \e^{1/4}\rho^\e_{t_k}(x) },
\quad	
	\Ge_2 := 
	\sup_{k\leq T\e^{-1}} \sup_{x\in[\e^b,L]}
	\curBK{	
		\e^{1/4}\absBK{ \calDe(\I_{t_k}(x),\I_0(x), t_k) }
		},
\end{align*}
we next show
\begin{enumerate*}[label=\itshape\roman*\upshape)]
\item $ \Ge_1 \to_\text{P} 0 $;
and \item $ \Ge_2 \to_\text{P} 0 $.
\end{enumerate*}

(\textit{i}) 
From \eqref{eq:I} and \eqref{eq:EMz}, we have 
$ 
	\I_{t_k}(x) = \anglebk{ \EM_{t_k}, \ind_{(-\infty,x]} } 
	= \anglebk{ \EMz_{t_k}, \ind_{(-\infty,x-\Xe_{(0)}(t_k)]} }.
$
Further,
by Proposition~\ref{prop:Xotig}, 
with an \ac{OP} we have $ \sup_{t\in[0,T]}\{|\Xe_{(0)}(t)|\}\leq 1 $,
so, with an \ac{OP}, 
$ \I_{t_k}(L) \leq \anglebk{ \EMz_{t_k}, \ind_{(-\infty,L+1]} } $ for all $ k\leq T\e^{-1} $.
Using this and the large deviations bound of $ \anglebk{ \EMz_{t_k}, \ind_{(-\infty,L+1]} } \sim \Pois(2\e^{-1/2}(L+1)) $,
we then conclude that
\begin{align}\label{eq:Ibd}
	\curBK{ \I_{\te_k}(L) \leq 4(L+1)\e^{-1/2}, \ \forall k\leq T\e^{-1} }
	\text{ holds with an \ac{OP}.}
\end{align}
Next, By Proposition~\ref{prop:Xotig},
with an \ac{OP}, for all $ x\in[\e^{b},L] $
and $ t\in[0,T] $,
we have $ \Xe_{(0)}(t) \leq \e^{b} \leq x $.
Consequently, by \eqref{eq:grd},  
$ |\rho^\e_t(x)| \leq Y_{\I_t(x)}(\e^{-1}t) $, with an \ac{OP}.
Combining this with \eqref{eq:Ibd},
we then conclude that, with an \ac{OP},
$
	\Ge_1 
	\leq 
	\sup_{k\leq T\e^{-1}} \sup_{|j|\leq 4(L+1)\e^{-1/2}} 
	\{ \e^{1/4} Y_{i}(t_k) \},
$
which clearly converges to zero in probability.

(\textit{ii})
By Proposition~\ref{prop:fluConv} and Proposition~\ref{prop:fluxii}\ref{enu:fluflug},
the process
$
	(t,x) \mapsto (\flug_t(x)-\flug_0(x)) \ind_{[\e^b,\infty)}(x)
$
converges weakly.
The latter, by \eqref{eq:flug}--\eqref{eq:I},
is equal to $ \e^{1/4}(\I_t(x)-\I_0(x)) \ind_{[\e^b,\infty)}(x) $.
From this, we conclude that, for any $ \mu'>1/4 $,
\begin{align}\label{eq:flux2:J}
	\lim_{a\to\infty} \lim_{\e\to 0}
	\Pro 
	\BK{ 
		\sup_{t\in[0,T]} \sup_{x\in[\e^{b},L]}  
		\absBK{\I_t(x)-\I_0(x)} \leq \e^{-\mu'}
	}
	=1. 
\end{align}
Fix arbitrary $ \mu'\in(1/4,1/2) $.
With $ \scrI_{\mu'}(T,L) $ defined as in \eqref{eq:muset},
combining \eqref{eq:Ibd} and \eqref{eq:flux2:J},
we arrive at
\begin{align*}
	\lim_{\e\to 0} 
	\Pro \BK{ \Ge_2 \leq \sup_{(j,j',k)\in\scrI_{\mu'}(T,L)} \curBK{ \e^{1/4}\absBK{ \calDe(j,j',\te_k) } }  } 
	= 1.
\end{align*}
From this and Lemma~\ref{lem:calD}, we conclude $ \Ge_2 \to_\text{P} 0 $.
\end{proof}

\begin{proof}[Proof of Proposition~\ref{prop:fluxii}\ref{enu:tagpptagp}]

Fixing $ L, T \geq 0 $ and $ b \in(0,1/4) $,
by \eqref{eq:locxi}--\eqref{eq:loczeta} and \eqref{eq:tagpptagp}, 
it suffices to show
\begin{align}\label{eq:flux3:red}
	\sup_{k \leq T\e^{-1}} \sup_{x\in[0,L]} 
	\absBK{ \e^{1/4} \calDe(\I_0(x+\e^{b}),\ie(x), t_k) }
	\xrightarrow[\text{P}]{} 0,
\end{align}
as $ \e\to 0 $.
As shown in the proof of Proposition~\ref{prop:fluxii}\ref{enu:flugtagpp},
this amounts to showing,
for some $ \mu'\in(0,1/2) $,
$
	\lim_{\e\to 0}
	\Pro \BK{ \absBK{\I_0(x+\e^b)-\ie(x)} \leq \e^{-\mu'}, \forall x\in[0,L] }
	 = 1.
$
By Markov's inequality, with $ b<1/4 $, this in turn follows from
\begin{align}\label{eq:flux3:claim:}
	\Ex\BK{ \sup_{x\in[0,L]} \e^{1/2-b}\absBK{\I_0(x+\e^b)-\ie(x)} }^2
	\leq
	C.
\end{align}
With $ \I_0(x')=\anglebk{ \EM_0, \ind_{[0,x]} } $ and $ \ie(x) := \floorbk{2\e^{-1}x} $,
we have
$ \e^{1/2-b}(\I_0(x+\e^b)-\ie(x)) = \e^{1/2-b} m^\e(x+\e^b) - 2 + \e^{1/2-b} \re $,
for some $ |\re| \leq 1 $ and for $ m^\e(x'):= \anglebk{ \EM_0, \ind_{[0,x']} } -2\e^{-1/2}x' $.
The process $ m^{\e}(\Cdot) $ is a martingale since $ \EM_0 \sim \PPPp $.
With $ b\in(0,1/4) $,
applying Doob's $ L^2 $ maximal inequality to $ m^\e(\Cdot) $,
we obtain $ \Ex(\sup_{x\in[0,L+1]} \{\e^{1/4}m^\e(x)\})^2 \leq C $,
thereby concluding \eqref{eq:flux3:claim:}.
\end{proof}

\bibliographystyle{abbrv}
\bibliography{DT-altas}

\begin{thebibliography}{10}

\bibitem{bass87}
R.~F. Bass and E.~Pardoux.
\newblock Uniqueness for diffusions with piecewise constant coefficients.
\newblock {\em Probab. Theory Related Fields}, 76(4):557--572, 1987.

\bibitem{brown71}
B.~M. Brown.
\newblock Martingale central limit theorems.
\newblock {\em The Annals of Mathematical Statistics}, 42(1):59--66, 1971.

\bibitem{cabezas15}
M.~Cabezas, A.~Dembo, A.~Sarantsev, and V.~Sidoravicius.
\newblock {B}rownian particles of rank-dependent drifts: out of equilibrium
  behavior.
\newblock {\em In preperation.}

\bibitem{chatterjee10}
S.~Chatterjee and S.~Pal.
\newblock A phase transition behavior for brownian motions interacting through
  their ranks.
\newblock {\em Probability theory and related fields}, 147(1-2):123--159, 2010.

\bibitem{chatterjee11}
S.~Chatterjee and S.~Pal.
\newblock A combinatorial analysis of interacting diffusions.
\newblock {\em Journal of Theoretical Probability}, 24(4):939--968, 2011.

\bibitem{dembo12}
A.~Dembo, M.~Shkolnikov, S.~Varadhan, and O.~Zeitouni.
\newblock Large deviations for diffusions interacting through their ranks.
\newblock {\em To appear Comm. Pure Appl. Math.}, 2012.

\bibitem{duerr85}
D.~D{\"u}rr, S.~Goldstein, and J.~L. Lebowitz.
\newblock Asymptotics of particle trajectories in infinite one-dimensional
  systems with collisions.
\newblock {\em Comm. Pure Appl. Math.}, 38(5):573--597, 1985.

\bibitem{fernholz02}
E.~R. Fernholz.
\newblock {\em Stochastic portfolio theory}.
\newblock Springer, 2002.

\bibitem{harris65}
T.~Harris.
\newblock Diffusion with" collisions" between particles.
\newblock {\em J. Appl. Probab.}, 2(2):323--338, 1965.

\bibitem{ichiba10}
T.~Ichiba and I.~Karatzas.
\newblock On collisions of brownian particles.
\newblock {\em The Annals of Applied Probability}, 20(3):951--977, 2010.

\bibitem{ichiba13}
T.~Ichiba, I.~Karatzas, and M.~Shkolnikov.
\newblock Strong solutions of stochastic equations with rank-based
  coefficients.
\newblock {\em Probability Theory and Related Fields}, 156(1-2):229--248, 2013.

\bibitem{ichiba11}
T.~Ichiba, V.~Papathanakos, A.~Banner, I.~Karatzas, R.~Fernholz, et~al.
\newblock Hybrid atlas models.
\newblock {\em The Annals of Applied Probability}, 21(2):609--644, 2011.

\bibitem{kallenberg02}
O.~Kallenberg.
\newblock {\em Foundations of modern probability}.
\newblock springer, 2002.

\bibitem{karatzas09}
I.~Karatzas and R.~Fernholz.
\newblock Stochastic portfolio theory: an overview.
\newblock {\em Handbook of numerical analysis}, 15:89--167, 2009.

\bibitem{landim98}
C.~Landim, S.~Olla, and S.~Volchan.
\newblock Driven tracer particle in one dimensional symmetric simple exclusion.
\newblock {\em Comm. Math. Phys.}, 192(2):287--307, 1998.

\bibitem{landim00}
C.~Landim and S.~B. Volchan.
\newblock Equilibrium fluctuations for a driven tracer particle dynamics.
\newblock {\em Stochastic Process. Appl.}, 85(1):139--158, 2000.

\bibitem{pal08}
S.~Pal and J.~Pitman.
\newblock One-dimensional brownian particle systems with rank-dependent drifts.
\newblock {\em Ann. Appl. Probab.}, pages 2179--2207, 2008.

\bibitem{pal14}
S.~Pal, M.~Shkolnikov, et~al.
\newblock Concentration of measure for brownian particle systems interacting
  through their ranks.
\newblock {\em Ann. Appl. Probab.}, 24(4):1482--1508, 2014.

\bibitem{revuz99}
D.~Revuz and M.~Yor.
\newblock {\em Continuous martingales and Brownian motion}, volume 293.
\newblock Springer, 1999.

\bibitem{shkolnikov11}
M.~Shkolnikov.
\newblock Competing particle systems evolving by interacting levy processes.
\newblock {\em Ann. Appl. Probab.}, 21(5):1911--1932, 2011.

\end{thebibliography}

\end{document}